\documentclass[11pt,reqno]{amsart}
\usepackage{amssymb,amsmath,amsthm,epsfig,times}
\usepackage{mathrsfs}
\usepackage{epstopdf}
\usepackage{pdfpages}
\numberwithin{equation}{section} 
\usepackage{verbatim}
\usepackage[pdftex]{hyperref}

\def\C{{\mathbb C}}

\def\N{{\mathbb N}}
\def\P{{\mathbb P}}

\def\R{{\mathbb R}}
\def\*S{{\mathbb S}}

\def\interior{\hspace{0.02in}
\begin{picture}(6,6)
\put (0,0){\line(1,0){6}}
\put (6,0){\line(0,1){6}}
\end{picture} \hspace{0.04in}}
\newcommand{\dis}{\displaystyle}
\newcommand{\cal}{\mathcal}

\newtheorem{theorem}{Theorem}
\newtheorem{proposition}{Proposition}[section]
\newtheorem{lemma}[proposition]{Lemma}

\begin{document}
\title[Explicit Hodge-type decomposition]
{Explicit Hodge-type decomposition on projective complete intersections.}

\author{Gennadi M. Henkin}
\address{Institut de Mathematiques \\ Universite Pierre et Marie Curie \\ 75252 BC247 Paris \\
Cedex 05 \\ France, and CEMI  Acad. Sc. \\ 117418 \\ Moscow, Russia}
\curraddr{}
\email{guennadi.henkin@imj-prg.fr}

\author{Peter L. Polyakov}
\address{University of Wyoming \\ Department of Mathematics \\ 1000 E University Ave
\\ Laramie, WY 82071}
\curraddr{}
\email{polyakov@uwyo.edu}
\thanks{Second author was partially supported by the NEUP program of
the Department of Energy}



\dedicatory{To Carlos Berenstein on occasion of his 70-th birthday}

\begin{abstract}
We construct an explicit homotopy formula for the $\bar\partial$-complex
on a reduced complete intersection subvariety $V\subset\C\P^n$.
This formula can be interpreted as an explicit Hodge-type decomposition for residual
currents on $V$.
\end{abstract}

\maketitle

\section{Introduction.}\label{Introduction}

The goal of the present article is to construct an explicit Hodge-type decomposition
for the $\bar\partial$-operator on complete intersection subvarieties of $\C\P^n$
and to obtain for those varieties a constructive version of the
classical theorem of Hodge \cite{Ho, Wy, Kd}:

\newtheorem{Hthm}{Hodge Theorem}
\renewcommand{\theHthm}{}
\begin{Hthm}\label{ClassicalHodge}
Let $V\subset \C\P^n$ be an algebraic manifold. Let $Z^{(p,q)}(V)$ be the space
of smooth $\bar\partial$-closed $(p,q)$-forms on $V$, and $B^{(p,q)}(V)$ - the space
of smooth $\bar\partial$-exact $(p,q)$- forms on $V$. Then
\begin{itemize}
\item[(i)] there exist a finite-dimensional projection operator
$L:\ Z^{(p,q)}(V) \to H^{(p,q)}(V)$ into the subspace
of real analytic $\bar\partial$-closed forms on $V$ and for $q>0$ a linear operator
$I:\ Z^{(p,q)}(V) \to C^{(p,q-1)}(V)$ such that for an arbitrary $\phi\in Z^{(p,q)}(V)$ the
following equality is satisfied
$$\phi=\bar\partial I[\phi]+L[\phi],$$
\item[(ii)] a form $\phi\in Z^{(p,q)}(V)$ is $\bar\partial$-exact iff $L[\phi]=0$.
\end{itemize}
\end{Hthm}

\indent
Theorems of this type have many applications, especially in algebraic geometry.
However, for some important applications there are at least two difficulties. The first difficulty
is caused by the non-constructiveness of the following remarkable Hodge's
statement: $V$ has to be equipped with an
hermitian metric, and then projection operator $L$ can be chosen to be orthogonal onto the subspace of harmonic $\bar\partial$-closed forms on $V$ (see \cite{Ho, GH, BDIP, V}). The
second difficulty is caused by too abstract formulations of necessary results for applications to varieties with singularities (see \cite{Dl, Gr, Ha, RR}).

\indent
The first difficulty has been overcome (rather recently) only for special cases
($\C\P^n$ and some flag manifolds) in \cite{HP1, Bn, Go, GSS, SS}.
An analytic technique for overcoming the second difficulty was initiated in \cite{HP2}
using an important theory of residual currents of Coleff and Herrera \cite{CH},
based on resolution of singularities of Hironaka \cite{Hi}.
In the present article we further develop our homotopy formulas for the $\bar\partial$-operator
from \cite{HP1} and combine them with the theory of residual currents to obtain a constructive
version of a Hodge-type decomposition for residual $\bar\partial$-cohomologies on complete
intersection subvarieties of $\C\P^n$.

\indent
The main result of the article is formulated in Theorem~\ref{HodgeRepresentation}
below. We notice that the decomposition obtained in this theorem, which explicitly depends
only on polynomials defining $V$, is new even in the case of a nonsingular curve in $\C\P^2$.

\indent
Before formulating this result we have to recap some of the definitions
from \cite{HP3}. Let $V$ be a complete intersection subvariety
\begin{equation}\label{Variety}
V=\left\{z\in \C\P^n:\ P_1(z)=\cdots=P_m(z)=0\right\}
\end{equation}
of dimension $n-m$ in $\C\P^n$ defined by a collection $\left\{P_k\right\}_{k=1}^m$
of homogeneous polynomials. Let
$$\left\{U_{\alpha}=\left\{z\in\C\P^n:\ z_{\alpha}\neq 0\right\}\right\}_{\alpha=0}^n$$
be the standard covering of $\C\P^n$, and let
$${\bf F}^{(\alpha)}(z)=\left[\begin{tabular}{c}
$F^{(\alpha)}_1(z)$\vspace{0.1in}\\
\vdots\vspace{0.1in}\\
$F^{(\alpha)}_m(z)$
\end{tabular}\right]
=\left[\begin{tabular}{c}
$P_1(z)/z_{\alpha}^{\deg P_1}$\vspace{0.1in}\\
\vdots\vspace{0.1in}\\
$P_m(z)/z_{\alpha}^{\deg P_m}$
\end{tabular}\right]$$
be collections of nonhomogeneous polynomials satisfying
$${\bf F}^{(\alpha)}(z)=A_{\alpha\beta}(z)\cdot{\bf F}^{(\beta)}(z)
=\left[\begin{tabular}{ccc}
$\left(z_{\beta}/z_{\alpha}\right)^{\deg P_1}$&$\cdots$&0\vspace{0.1in}\\
\vdots&$\ddots$&\vdots\vspace{0.1in}\\
0&$\cdots$&$\left(z_{\beta}/z_{\alpha}\right)^{\deg P_m}$
\end{tabular}\right]
\cdot{\bf F}^{(\beta)}(z)$$
on $U_{\alpha\beta}=U_{\alpha}\cap U_{\beta}$.\\
\indent
Following \cite{Gr} and \cite{Ha} we consider a line bundle ${\cal L}$ on $V$
with transition functions
$$l_{\alpha\beta}(z)=\det A_{\alpha\beta}
=\left(\frac{z_{\beta}}{z_{\alpha}}\right)^{\sum_{k=1}^m\deg P_k}$$
on $U_{\alpha\beta}$ and the {\it dualizing bundle} on a complete intersection subvariety $V$
\begin{equation}\label{Dualizing}
\omega^{\circ}_V=\omega_{\C\P^n}\otimes{\cal L},
\end{equation}
where $\omega_{\C\P^n}$ is the canonical bundle on $\C\P^n$.\\
\indent
For $q=1,\dots,n-m$ we denote by ${\cal E}^{(n,n-m-q)}\left(V,{\cal L}\right)
={\cal E}^{(0,n-m-q)}\left(V,\omega^{\circ}_V\right)$ the space
of $C^{\infty}$ differential forms of bidegree $(n,n-m-q)$ with coefficients in ${\cal L}$,
i.e. the space of collections of forms
$$\left\{\gamma_{\alpha}\in
{\cal E}^{(n,n-m-q)}\left(U_{\alpha}\right)\right\}_{\alpha=0}^n$$
satisfying
\begin{equation}\label{Transition}
\gamma_{\alpha}=l_{\alpha\beta}\cdot\gamma_{\beta}
+\sum_{k=1}^mF^{(\alpha)}_k\cdot\gamma^{\alpha\beta}_k\
\mbox{on}\ U_{\alpha}\cap U_{\beta}.
\end{equation}
\indent
Then following \cite{CH, HP3, Pa1} we define residual currents and $\bar\partial$-closed
residual currents on $V$. By a residual current of homogeneity zero $\phi\in C_R^{(0,q)}(V)$
we call a collection $\left\{\Phi_{\alpha}^{(0,q)}\right\}_{\alpha=0}^n$
of $C^{\infty}$ differential forms satisfying equalities
\begin{equation}\label{ResidualCurrent}
\Phi_{\alpha}=\Phi_{\beta}+\sum_{k=1}^mF^{(\alpha)}_k
\cdot\Omega^{(\alpha\beta)}_k\ \mbox{on}\ U_{\alpha}\cap U_{\beta},
\end{equation}
acting on $\gamma\in {\cal E}^{(n,n-m-q)}\left(V,{\cal L}\right)$ by the formula
\begin{equation}\label{GlobalCurrent}
\langle\phi,\gamma\rangle=\sum_{\alpha}
\int_{U_{\alpha}}\vartheta_{\alpha}\gamma_{\alpha}\wedge\Phi_{\alpha}
\bigwedge_{k=1}^m\bar\partial\frac{1}{F^{(\alpha)}_k}
\stackrel{\text{def}}{=}\lim_{t\to 0}\sum_{\alpha}\int_{T^{\epsilon(t)}_{\alpha}}
\vartheta_{\alpha}\frac{\gamma_{\alpha}\wedge\Phi_{\alpha}}{\prod_{k=1}^m F^{(\alpha)}_k},
\end{equation}
where $\left\{\vartheta_{\alpha}\right\}_{\alpha=0}^n$ is a partition of unity
subordinate to the covering $\left\{U_{\alpha}\right\}_{\alpha=0}^n$, and
the limit in the right-hand side of \eqref{GlobalCurrent} is taken along an admissible path
in the sense of Coleff-Herrera \cite{CH}, i.e. an analytic map
$\epsilon:[0,1]\to \R^m$ satisfying conditions
\begin{equation}\label{admissible}
\begin{cases}
\lim_{t\to 0}\epsilon_m(t)=0,\\
{\dis \lim_{t\to 0}\frac{\epsilon_j(t)}{\epsilon^l_{j+1}(t)}=0,\
\mbox{for any}\ l\in\N }\ \text{and}\ j=1,\dots,m-1,
\end{cases}
\end{equation}
and
\begin{equation}\label{Tube}
T^{\epsilon(t)}_{\alpha}
=\left\{z\in U_{\alpha}:\ \left|F^{(\alpha)}_k(z)\right|=\epsilon_k(t)\ \text{for}\ k=1,\dots,m\right\}.
\end{equation}
Condition \eqref{admissible}, though looking technical, is essential for the existence of the limit
in the right-hand side of \eqref{GlobalCurrent}, and can not be replaced by a simpler condition
$\epsilon_j(t)\to 0$, $t\to 0$, $j=1,\ldots,m$, as was shown by Passare and Tsikh in \cite{PT}.\\
\indent
A residual current $\phi$ we call $\bar\partial$-closed
$\left(\text{denoted}\ \phi\in Z_R^{(0,q)}(V)\right)$, if it satisfies the following condition
\begin{equation}\label{Closed}
\bar\partial\Phi_{\alpha}
=\sum_{k=1}^mF^{(\alpha)}_k\cdot\Omega^{(\alpha)}_k\ \mbox{on}\ U_{\alpha}.
\end{equation}

\indent
In Theorem~\ref{HodgeRepresentation} below we prove the existence of an explicit
Hodge-type representation formula for $\bar\partial$-closed residual currents
and its main properties. For simplification of formulation
and of the exposition below we assume existence of holomorphic functions
$g_{\alpha}\in H(U_{\alpha})$ for $\alpha\in (0,\dots,n)$ satisfying
\begin{equation}\label{gFunctions}
\begin{array}{ll}
(a)\quad V^{\prime}_{\alpha}=\left\{z\in U_{\alpha}: F^{(\alpha)}_1(z)=\cdots=F^{(\alpha)}_m(z)
=g_{\alpha}(z)=0\right\}\ \text{is a complete intersection in}\ U_{\alpha},\vspace{0.1in}\\
(b)\quad \left(V\cap U_{\alpha}\right)\setminus V^{\prime}_{\alpha}\
\text{is a submanifold in}\ U_{\alpha}.
\end{array}\end{equation}
Existence of such functions is a corollary of the local description of analytic sets (see \cite{RS}).

\begin{theorem}\label{HodgeRepresentation} Let $V\subset \C\P^n$ be a reduced complete
intersection subvariety as in \eqref{Variety}.\\
\indent
Then
\begin{itemize}
\item[(i)]
there exist an explicit finite-dimensional projection operator
(see formula \eqref{LResidue} below)
$$L_{n-m}: Z_R^{(0,n-m)}\left(V\right)\to Z_R^{(0,n-m)}\left(V\right)$$
into the subspace of real analytic $\bar\partial$-closed residual currents and explicit
linear operators (see formula \eqref{SolutionFormula} below)
$$I_q: Z_R^{(0,q)}\left(V\right)\to C^{(0,q-1)}\left(V\right)$$
into the spaces of currents on $V$
for $q=1,\dots,n-m$, so that the following equality is satisfied for an arbitrary
$\phi\in Z_R^{(0,q)}\left(V\right)$:
\begin{equation}\label{HodgeHomotopy}
\phi=\bar\partial I_q[\phi]+L_q[\phi],
\end{equation}
\item[(ii)]
for $q=1,\dots,n-m-1$ we have $L_q=0$, and therefore $I_q[\phi]$ is a current-solution
of equation $\bar\partial\psi=\phi$, which is a residual current on
$V\setminus \bigcup_{\alpha} V^{\prime}_{\alpha}$
defined by the forms in $C^{\infty}\left(U_{\alpha}\setminus V^{\prime}_{\alpha}\right)$,
\item[(iii)]
a $\bar\partial$-closed residual current $\phi\in Z_R^{(0,n-m)}\left(V\right)$
of homogeneity zero is $\bar\partial$-exact, i.e. there exists a current
$\psi\in C^{(0,n-m-1)}\left(V\right)$ such that $\phi=\bar\partial\psi$, iff
\begin{equation}\label{HodgeCondition}
L_{n-m}[\phi]=0.
\end{equation}
\end{itemize}
\end{theorem}

{\bf Remark 1.} We interpret formula \eqref{HodgeHomotopy} as equality of currents,
which are principal values of the residues of Coleff and Herrera taken along admissible paths.
Precise definitions and explanations are given in the end of Section~\ref{Formulas}
and in Section~\ref{Proof}. Such interpretation with application to explicit solvability
of $\bar\partial$-equation on Stein reduced complete intersections in
pseudoconvex domains was introduced in \cite{HP2}, motivated by the works of Coleff,
Herrera, and Lieberman \cite{CH, HL}.
In \cite{AS1, AS2} such interpretation was used to obtain similar
solvability of $\bar\partial$-equation on reduced pure-dimensional Stein spaces.
In the present article we use this interpretation in the problem of constructing an explicit 
Hodge-type decomposition of $\bar\partial$-closed residual currents on reduced, compact,
complete intersection subvarieties in $\C\P^n$ with nontrivial $\bar\partial$-cohomologies
of highest degree. An important feature of the obtained decomposition is condition
\eqref{HodgeCondition}, which is similar to condition (ii) in the Hodge Theorem,
but with explicit integral operator $L_{n-m}$. Another important feature of decomposition
\eqref{HodgeHomotopy} is the real analyticity of the form $L_{n-m}[\phi]$
in some neighborhood of $V$ even for the case of singular reduced complete intersections.

{\bf Remark 2.} Works of Passare \cite{Pa1, Pa2}, and of Berenstein, Gay, and Yger \cite{BGY},
based on fundamental results of Atiyah \cite{At}, J. Bernstein, S. Gelfand \cite{BG},
and J. Bernstein \cite{Be} lead to the following simplified version of the original
Colleff-Herrera-Lieberman residue formula
$$\left\langle\Phi,\gamma\right\rangle
=\lim_{\begin{array}{ll}
\lambda_1,\dots,\lambda_m\to 0\\
\text{Re}\lambda_j>0,\ j=1,\dots,m
\end{array}}
\sum_{\alpha}\vartheta_{\alpha}\frac{\gamma_{\alpha}
\wedge\bar\partial\left|F_1^{(\alpha)}\right|^{2\lambda_1}\wedge\cdots
\wedge\bar\partial\left|F_m^{(\alpha)}\right|^{2\lambda_m}\wedge\Phi_{\alpha}}
{\prod_{j=1}^mF_j^{(\alpha)}}.$$

\indent
It was used in \cite{BGY} for the following division and interpolation problem:\\
\indent
{\it for given holomorphic functions $f_1, \dots, f_p$ and an arbitrary holomorphic
function $f$ on a Stein variety $V$ find an explicit representation
$$f = f_1g_1 +\cdots + f_pg_p + h$$
with $g_j$ being holomorphic on $V$, so that the remainder $h$ vanishes
on $V$ iff $f$ belongs to the ideal generated by $f_1, \dots, f_p$.}

{\bf Remark 3.} In the work in preparation we are planning to give two interrelated simple applications of Theorem~\ref{HodgeRepresentation}:
\begin{itemize}
\item[(i)]
construction of an explicit Hodge-type decomposition on complex curves in $\C\P^3$,\vspace{0.1in}
\item[(ii)]
construction of explicit Green's functions for solutions of inverse conductivity problem on bordered
surfaces in $\R^3$.
\end{itemize}

{\bf Remark 4.} In the future we plan to extend the result of
Theorem~\ref{HodgeRepresentation} to the case of locally complete intersections in $\C\P^n$
with $n\geq 3$, which might be considered as a natural level of generality for explicit formulas, as
implied by Hartshorne \cite{Ha}.

\section{Integral formulas on domains in projective spaces.}
\label{Formulas}

\indent
In this section we construct a Cauchy-Weil-Leray type integral formula for differential
forms on a domain $U$ in $\C\P^n$. We start with the Koppelman-type formula
from \cite{HP1} (Proposition 1.2) and \cite{He} (Theorem 3.2) going back to Moisil \cite{Mo},
Fueter \cite{F}, Bochner \cite{Bo}, Martinelli \cite{Ma}.
This formula is a modification for the case of $\C\P^n$ of the original Koppelman formula
announced by Koppelman in \cite{Kp} (1967). The first
complete proof of Koppelman's formula was given in the Polyakov's paper \cite{Po} (06.1970),
where it was used to obtain a Weil-type integral formula \cite{Wi} for differential forms
on analytic polyhedra, while in the papers of Lieb \cite{Li} (07.1970)
and \O vrelid \cite{O} (11.1970)
Koppelman's formula was used to obtain an integral formula of Leray-type \cite{Le}
for differential forms on strongly pseudoconvex domains. In the present article
we use formulas of both types: Weil-type formula for a tubular neighborhood
of a subvariety in $\C\P^n$ and Leray-type formula for the unit sphere
$\*S^{2n+1}(1)\subset \C^{n+1}$.\\
\indent
In \cite{HP1} we identified the forms on $\C\P^n$ with their lifts to $\*S^{2n+1}(1)$ satisfying
appropriate homogeneity conditions and constructed integral formulas for the lifted forms.
The proposition below is a reformulation of Proposition 1.2 from \cite{HP1}.

\begin{proposition}\label{Bochner} Let $\left\{P_k\right\}_1^m$ be homogeneous polynomials
defining the variety $V$ as in \eqref{Variety},
let $\epsilon=\left(\epsilon_1,\dots,\epsilon_m\right)$,
and let $\Phi^{(0,q)}$ be a form of homogeneity zero on the domain
\begin{equation}\label{Uepsilon}
U^{\epsilon}=\left\{z\in \*S^{2n+1}(1):\ \left|P_k(z)\right|<\epsilon_k\hspace{0.1in} \text{for}\
k=1,\dots,m\right\}.
\end{equation}
\indent
Then the following equality is satisfied for $z\in U^{\epsilon}$
\begin{equation}\label{MartinelliFormula}
\Phi^{(0,q)}(z)=\bar\partial_zJ_q^{\epsilon}\left[\Phi\right](z)
+J_{q+1}^{\epsilon}\left[\bar\partial\Phi\right](z)+K_q^{\epsilon}\left[\Phi\right](z),
\end{equation}
with
\begin{equation*}
J_q^{\epsilon}\left[\Psi\right](z)
=-\frac{n!}{(2\pi i)^{n+1}}\int_{U^{\epsilon}\times[0,1]}
\Psi(\zeta)\wedge\omega^{\prime}_{q-1}\left((1-\lambda)\frac{\bar z}{B^*(\zeta,z)}
+\lambda\frac{\bar\zeta}{B(\zeta,z)}\right)\wedge\omega(\zeta),
\end{equation*}
and
\begin{equation*}
K_q^{\epsilon}\left[\Psi\right](z)
=-\frac{n!}{(2\pi i)^{n+1}}\int_{bU^{\epsilon}\times[0,1]}
\Psi(\zeta)\wedge\omega^{\prime}_q\left((1-\lambda)\frac{\bar z}
{B^*(\zeta,z)}+\lambda\frac{\bar\zeta}{B(\zeta,z)}\right)\wedge\omega(\zeta),
\end{equation*}
where
$$B^*(\zeta,z)=\sum_{j=0}^n{\bar z}_j\cdot\left(\zeta_j-z_j\right),
\hspace{0.2in}B(\zeta,z)=\sum_{j=0}^n{\bar\zeta}_j\cdot\left(\zeta_j-z_j\right),$$
$$\omega(\zeta)=d\zeta_0\wedge d\zeta_1\wedge\cdots\wedge d\zeta_n,\hspace{0.2in}
\omega^{\prime}(\eta)=\sum_{k=0}^n(-1)^{k-1}\eta_k\bigwedge_{j\neq k}d\eta_j$$
and $\omega^{\prime}_q$ is the $(0,q)$-component with respect to $z$ of the form
$\omega^{\prime}$.
\end{proposition}

\indent
We will transform the right-hand side of equality \eqref{MartinelliFormula} into a Cauchy-Weil-Leray
type formula. For this transformation we need the following Weil-type lemma.

\begin{lemma}\label{WeilCoefficients} Let $P(\zeta)$ be a homogeneous polynomial of variables
$\zeta_0,\dots,\zeta_n$ of degree $d$. Then there exist polynomials
$\left\{Q^i(\zeta,z)\right\}_{i=0}^n$ satisfying:
\begin{equation}\label{HomogeneityConditions}
\left\{\begin{array}{ll}
P(\zeta)-P(z)=\sum_{i=0}^nQ^i(\zeta,z)\cdot\left(\zeta_i-z_i\right),\vspace{0.1in}\\
Q^i(\lambda\zeta,\lambda z)=\lambda^{d-1}\cdot Q^i(\zeta,z)\ \mbox{for}\ \lambda\in\C.
\end{array}\right.
\end{equation}
\end{lemma}
\begin{proof} We notice that it suffices to prove the lemma for homogeneous monomials.
We prove the lemma for homogeneous monomials by induction with respect to the number
of variables. Using the one-variable equality
\begin{equation}\label{OneVariableEquality}
\zeta^d-z^d=\left(\zeta-z\right)\cdot\left(\sum_{j=0}^{d-1}\zeta^{d-1-j}\cdot z^{j}\right)
\end{equation}
we obtain the statement of the Lemma for an arbitrary monomial depending only on one variable.\\
\indent
To prove the step of induction we consider a monomial $\zeta_0^{d_0}\cdots\zeta_k^{d_k}$ with
$k\geq 1$ and $\sum_{j=0}^k d_j=d$. Then we obtain the following equality
$$\zeta_0^{d_0}\cdots\zeta_k^{d_k}-z_0^{d_0}\cdots z_k^{d_k}
=\left(\zeta_0^{d_0}-z_0^{d_0}\right)\cdot\zeta_1^{d_1}\cdots\zeta_k^{d_k}
+z_0^{d_0}\cdot\left(\zeta_1^{d_1}\cdots\zeta_k^{d_k}
-z_1^{d_1}\cdots z_k^{d_k}\right).$$
\indent
Using equality \eqref{OneVariableEquality} for the first term of the right-hand side
of equality above we obtain
$$Q^0(\zeta,z)=\left(\sum_{j=0}^{d_0-1}\zeta_0^{d_0-1-j}\cdot z_0^{j}\right)
\cdot\zeta_1^{d_1}\cdots\zeta_k^{d_k}.$$
Using then the inductive assumption for the polynomial
$$\zeta_1^{d_1}\cdots\zeta_k^{d_k}-z_1^{d_1}\cdots z_k^{d_k}$$
we obtain the existence of polynomials $\left\{q^i(\zeta,z)\right\}_{i=1}^n$ satisfying
conditions \eqref{HomogeneityConditions}. Therefore, defining for $i=1,\dots,n$
$$Q^i(\zeta,z)=z_0^{d_0}\cdot q^i(\zeta,z)$$
we obtain the necessary coefficients for a monomial in $k+1$ variables.
\end{proof}

\indent
The integrals in the sought formula will be taken over a special chain
\begin{equation}\label{C_Chain}
{\cal C}^{\epsilon}=\sum_{|J|\geq 1}\Gamma^{\epsilon}_J\times\Delta_J,
\end{equation}
where
$J=\left(j_1,\dots,j_p\right)$ is a multiindex with $|J|=p\leq m$,
$$\Gamma^{\epsilon}_J=\left\{\zeta\in \*S^{2n+1}(1):\ |P_j(\zeta)|=\epsilon_j\
\mbox{for}\ j\in J,\ |P_k(\zeta)|<\epsilon_k\ \mbox{for}\ k\notin J\right\},$$
$$\Delta_J=\left\{\lambda,\mu_{j_1},\dots,\mu_{j_p}\in\R^{p+1}:\
\lambda+\sum_{i=1}^p\mu_{j_i}\leq 1\right\}.$$
\indent
The boundary of chain ${\cal C}^{\epsilon}$ is the chain
\begin{equation*}
{\cal B}^{\epsilon}=-\sum_{j=1}^m\Gamma^{\epsilon}_j\times\Lambda
+\sum_{|J|\geq 1}\left((-1)^{|J|-1}\Gamma^{\epsilon}_J\times\Delta^{\prime}_J
+\Gamma^{\epsilon}_J\times\Lambda_J\right),
\end{equation*}
where
$$\Lambda=[0,1],$$
$$\Delta^{\prime}_J=\left\{\mu_{j_1},\dots,\mu_{j_p}\in\R^p:\
\sum_{i=1}^p\mu_{j_i}\leq 1\right\},$$
$$\Lambda_J=\left\{\lambda,\mu_{j_1},\dots,\mu_{j_p}\in\R^{p+1}:\
\lambda+\sum_{i=1}^p\mu_{j_i}=1\right\}.$$
\indent
In the following proposition we construct a Cauchy-Weil-Leray type formula on
$\epsilon$-neighborhoods of complete intersection subvarieties in $\C\P^n$.

\begin{proposition}\label{Formula} Let
$$V=\left\{z\in \C\P^n:\ P_1(z)=\cdots =P_m(z)=0\right\}$$
be a complete intersection subvariety in $\C\P^n$ of dimension $n-m$, and
let $\Phi^{(0,q)}$ be a differential form on an open neighborhood $U\supset V$.\\
\indent
Then for $U^{\epsilon}$ as in \eqref{Uepsilon} and arbitrary $z\in U^{\epsilon}$
the following equality holds
\begin{equation}\label{WeilLerayFormula}
\Phi(z)=\bar\partial_z I_q^{\epsilon}\left[\Phi\right](z)
+I_{q+1}^{\epsilon}\left[\bar\partial\Phi\right](z)+L_q^{\epsilon}\left[\Phi\right](z),
\end{equation}
where
\begin{multline}\label{IOperator}
I_q^{\epsilon}\left[\Phi\right](z)=-\frac{n!}{(2\pi i)^{n+1}}\int_{U^{\epsilon}\times[0,1]}
\Phi(\zeta)\wedge\omega^{\prime}_{q-1}\left((1-\lambda)\frac{\bar z}{B^*(\zeta,z)}
+\lambda\frac{\bar\zeta}{B(\zeta,z)}\right)\wedge\omega(\zeta)\\
+\frac{n!}{(2\pi i)^{n+1}}\sum_{|J|\geq 1}\int_{\Gamma^{\epsilon}_J\times\Delta_J}
\Phi(\zeta)\wedge 
\omega^{\prime}_{q-1}\left((1-\lambda-\sum_{k=1}^m\mu_k)
\frac{\bar z}{B^*(\zeta,z)}\right.\\
\left.+\lambda\frac{\bar\zeta}{B(\zeta,z)}
+\sum_{k=1}^m\mu_k\frac{Q_k(\zeta,z)}{P_k(\zeta)-P_k(z)}\right)
\wedge\omega(\zeta),
\end{multline}
and
\begin{multline}\label{LOperator}
L_q^{\epsilon}\left[\Phi\right](z)
=\sum_{|J|=n-q}(-1)^{|J|-1}\frac{n!}{(2\pi i)^{n+1}}
\int_{\Gamma^{\epsilon}_J\times\Delta^{\prime}_J}
\Phi(\zeta)\wedge\omega^{\prime}_q\left((1-\sum_{k=1}^m\mu_k)
\frac{\bar z}{B^*(\zeta,z)}\right.\\
\left.+\sum_{k=1}^m\mu_k\frac{Q_k(\zeta,z)}{P_k(\zeta)-P_k(z)}\right)
\wedge\omega(\zeta),
\end{multline}
with coefficients $\left\{Q_k^i\right\}_{k=1,\dots,m}^{i=0,\dots,n}$ satisfying conditions
\eqref{HomogeneityConditions} from Lemma~\ref{WeilCoefficients}.\\
\indent
The forms defined by \eqref{IOperator} and \eqref{LOperator} on $U^{\epsilon}$
admit the descent onto a neighborhood of $V$ in $\C\P^n$.
\end{proposition}
\begin{proof} Applying the Stokes' formula to the form
\begin{equation*}
\Phi(\zeta)\wedge\omega^{\prime}_q\left((1-\lambda-\sum_{k=1}^m\mu_k)
\frac{\bar z}{B^*(\zeta,z)}+\lambda\frac{\bar\zeta}{B(\zeta,z)}
+\sum_{k=1}^m\mu_k\frac{Q_k(\zeta,z)}{P_k(\zeta)-P_k(z)}\right)\wedge\omega(\zeta)
\end{equation*}
we obtain equality
\begin{multline*}
\int_{bU^{\epsilon}\times[0,1]}\Phi(\zeta)
\wedge\omega^{\prime}_q\left((1-\lambda)\frac{\bar z}
{B^*(\zeta,z)}+\lambda\frac{\bar\zeta}{B(\zeta,z)}\right)\wedge\omega(\zeta)\\
=\sum_{j=1}^m\int_{\Gamma^{\epsilon}_j\times\Lambda}\Phi(\zeta)
\wedge\omega^{\prime}_q\left((1-\lambda)\frac{\bar z}
{B^*(\zeta,z)}+\lambda\frac{\bar\zeta}{B(\zeta,z)}\right)\wedge\omega(\zeta)\\
=\sum_{|J|\geq 1}(-1)^{|J|-1}\int_{\Gamma^{\epsilon}_J\times\Delta^{\prime}_J}
\Phi(\zeta)\wedge\omega^{\prime}_q\left((1-\sum_{k=1}^m\mu_k)\frac{\bar z}{B^*(\zeta,z)}
+\sum_{k=1}^m\mu_k\frac{Q_k(\zeta,z)}{P_k(\zeta)-P_k(z)}\right)\wedge\omega(\zeta)\\
+\sum_{|J|\geq 1}\int_{\Gamma^{\epsilon}_J\times\Lambda_J}
\Phi(\zeta)\wedge\omega^{\prime}_q\left((1-\sum_{k=1}^m\mu_k)\frac{\bar\zeta}
{B(\zeta,z)}+\sum_{k=1}^m\mu_k\frac{Q_k(\zeta,z)}{P_k(\zeta)-P_k(z)}\right)
\wedge\omega(\zeta)\\
-\sum_{|J|\geq 1}\int_{\Gamma^{\epsilon}_J\times\Delta_J}
\bar\partial\Phi(\zeta)
\wedge\omega^{\prime}_q\left((1-\lambda-\sum_{k=1}^m\mu_k)\frac{\bar z}{B^*(\zeta,z)}
+\lambda\frac{\bar\zeta}{B(\zeta,z)}+\sum_{k=1}^m\mu_k\frac{Q_k(\zeta,z)}{P_k(\zeta)-P_k(z)}\right)
\wedge\omega(\zeta)\\
+(-1)^{q+1}\sum_{|J|\geq 1}\int_{\Gamma^{\epsilon}_J\times\Delta_J}
\Phi(\zeta)\wedge d_{\zeta,\lambda,\mu}
\omega^{\prime}_q\left((1-\lambda-\sum_{k=1}^m\mu_k)\frac{\bar z}{B^*(\zeta,z)}\right.\\
+\left.\lambda\frac{\bar\zeta}{B(\zeta,z)}+\sum_{k=1}^m\mu_k\frac{Q_k(\zeta,z)}{P_k(\zeta)-P_k(z)}\right)
\wedge\omega(\zeta).
\end{multline*}
\indent
Then using equality
\begin {equation}\label{domega}
d_{\zeta,\lambda,\mu}\omega^{\prime}_{r}(\eta)\wedge\omega(\zeta)
+ \bar\partial_z \omega^{\prime}_{r-1}(\eta)\wedge\omega(\zeta) = 0
\hspace{0.3in} (r=1,\dots,n)
\end{equation}
for
$$\left\{\eta_j=(1-\lambda-\sum_{k=1}^m\mu_k)\frac{\bar z_j}
{B^*(\zeta,z)}+\lambda\frac{\bar\zeta_j}{B(\zeta,z)}
+\sum_{k=1}^m\mu_k\frac{Q_k(\zeta,z)}{P_k(\zeta)-P_k(z)}\right\}_{j=0}^n$$
we transform the equality above for $n\geq 2$ into
\begin{multline}\label{BoundaryIntegral}
\int_{bU^{\epsilon}\times[0,1]}\Phi(\zeta)
\wedge\omega^{\prime}_{q-1}\left((1-\lambda)\frac{\bar z}
{B^*(\zeta,z)}+\lambda\frac{\bar\zeta}{B(\zeta,z)}\right)\wedge\omega(\zeta)\\
=\sum_{|J|\geq 1}(-1)^{|J|-1}
\int_{\Gamma^{\epsilon}_J\times\Delta^{\prime}_J}\Phi(\zeta)
\wedge\omega^{\prime}_q\left((1-\sum_{k=1}^m\mu_k)\frac{\bar z}{B^*(\zeta,z)}
+\sum_{k=1}^m\mu_k\frac{Q_k(\zeta,z)}{P_k(\zeta)-P_k(z)}\right)\wedge\omega(\zeta)\\
+\sum_{|J|\geq 1}\int_{\Gamma^{\epsilon}_J\times\Lambda_J}\Phi(\zeta)
\wedge\omega^{\prime}_q\left((1-\sum_{k=1}^m\mu_k)\frac{\bar\zeta}
{B(\zeta,z)}+\sum_{k=1}^m\mu_k\frac{Q_k(\zeta,z)}{P_k(\zeta)-P_k(z)}\right)
\wedge\omega(\zeta)\\
-\sum_{|J|\geq 1}\int_{\Gamma^{\epsilon}_J\times\Delta_J}
\bar\partial\Phi(\zeta)
\wedge\omega^{\prime}_q\left((1-\lambda-\sum_{k=1}^m\mu_k)\frac{\bar z}{B^*(\zeta,z)}
+\lambda\frac{\bar\zeta}{B(\zeta,z)}+\sum_{k=1}^m\mu_k\frac{Q_k(\zeta,z)}{P_k(\zeta)-P_k(z)}\right)
\wedge\omega(\zeta)\\
+\sum_{|J|\geq 1}\int_{\Gamma^{\epsilon}_J\times\Delta_J}
\Phi(\zeta)\wedge 
\bar\partial_z\omega^{\prime}_{q-1}\left((1-\lambda-\sum_{k=1}^m\mu_k)
\frac{\bar z}{B^*(\zeta,z)}\right.\\
\left.+\lambda\frac{\bar\zeta}{B(\zeta,z)}
+\sum_{k=1}^m\mu_k\frac{Q_k(\zeta,z)}{P_k(\zeta)-P_k(z)}\right)
\wedge\omega(\zeta),
\end{multline}
and finally obtain from \eqref{MartinelliFormula} equality \eqref{WeilLerayFormula} with
\begin{multline*}
L_q^{\epsilon}\left[\Phi\right](z)=\sum_{|J|\geq 1}(-1)^{|J|-1}\frac{n!}{(2\pi i)^{n+1}}
\int_{\Gamma^{\epsilon}_J\times\Delta^{\prime}_J}\Phi(\zeta)
\wedge\omega^{\prime}_q\left((1-\sum_{k=1}^m\mu_k)\frac{\bar z}{B^*(\zeta,z)}\right.\\
\left.+\sum_{k=1}^m\mu_k\frac{Q_k(\zeta,z)}{P_k(\zeta)-P_k(z)}\right)\wedge\omega(\zeta)\\
+\sum_{|J|\geq 1}\frac{n!}{(2\pi i)^{n+1}}\int_{\Gamma^{\epsilon}_J\times\Lambda_J}\Phi(\zeta)
\wedge\omega^{\prime}_q\left((1-\sum_{k=1}^m\mu_k)\frac{\bar\zeta}{B(\zeta,z)}
+\sum_{k=1}^m\mu_k\frac{Q_k(\zeta,z)}{P_k(\zeta)-P_k(z)}\right)\wedge\omega(\zeta).
\end{multline*}
\indent
Then we notice that because of the holomorphic dependence on $z$ we have for $q\geq 1$ the equality
$$\omega^{\prime}_q\left((1-\sum_{k=1}^m\mu_k)\frac{\bar\zeta}
{B(\zeta,z)}+\sum_{k=1}^m\mu_k\frac{Q_k(\zeta,z)}{P_k(\zeta)-P_k(z)}\right)=0.$$
\indent
Since the dimension of $\Gamma^{\epsilon}_J$ is equal to $2n+1-|J|$ and the form $\Phi$ has
$q$ differentials of the form $d{\bar\zeta}$, we conclude that the only terms in the first
sum of the right-hand side of the formula for ${\dis L_q^{\epsilon}\left[\Phi\right] }$
that have a nonzero contribution are the terms with
$$|J|=n-q.$$
\indent
From the last two observations we obtain formula \eqref{LOperator} for
${\dis L_q^{\epsilon}\left[\Phi\right] }$.\\
\indent
The fact that the forms ${\dis L_q^{\epsilon}\left[\Phi\right] }$
and ${\dis I_q^{\epsilon}\left[\Phi\right] }$ have homogeneity zero, as the form $\Phi$,
follows from the homogeneity properties of the functions $B(\zeta,z)$ and
$B^*(\zeta,z)$ and from the homogeneity property \eqref{HomogeneityConditions} of the coefficients $Q_s^i(\zeta,z)$.
\end{proof}
\indent
We interpret formula \eqref{WeilLerayFormula} as a formula for residual currents
\begin{equation}\label{CurrentFormula}
\left\langle\phi,\gamma\right\rangle=\pm\left\langle
I_q^{\epsilon}\left[\phi\right],\bar\partial\gamma\right\rangle
+\left\langle I_{q+1}^{\epsilon}\left[\bar\partial\phi\right],\gamma\right\rangle
+\left\langle L_q^{\epsilon}\left[\phi\right],\gamma\right\rangle,
\end{equation}
where all terms in the right-hand side are understood as residual currents, i.e. for example
for an arbitrary $\gamma\in {\cal E}_c^{(n,n-m-q)}\left(V,{\cal L}\right)$ with support in
$U_{\alpha}$ we mean
\begin{equation}\label{preL-Operator}
\left\langle L_q^{\epsilon}\left[\phi\right],\gamma\right\rangle
=\lim_{\tau\to 0}\int_{T^{\delta(\tau)}_{\alpha}}
\gamma(z)\wedge\frac{L_q^{\epsilon}\left[\phi\right](z)}{\prod_{k=1}^m F^{(\alpha)}_k(z)},
\end{equation}
where we denote by $L_q^{\epsilon}\left[\phi\right](z)$ the descent of this form
onto $\C\P^n$.\\
\indent
Formula \eqref{CurrentFormula} is a preliminary form of the Hodge-type decomposition
formula for $\bar\partial$-closed residual currents on $V$. In what follows we will consider the limits of
the terms in the right-hand side of \eqref{WeilLerayFormula} as $\epsilon\to 0$, and interpret
the limit of operator $I_q^{\epsilon}\left[\Phi\right]$ as a solution operator on $V$ and the limit of
$L_q^{\epsilon}\left[\Phi\right]$ as a Hodge-type projection operator.

\section{Hodge-type projection.}
\label{LTransformation}

\indent
In this section we transform formula \eqref{LOperator} into a residual form by considering
the limit of $L_q^{\epsilon}$ as $\epsilon\to 0$.
We perform this transformation in several steps. First we observe that the only nonzero terms in this
formula are those that have $q=n-|J|$. But for subvariety $V$ we have $|J|\leq m$,
and therefore operator $L_q^{\epsilon}$ contains nonzero integrals only for $q\geq n-m$.
On the other hand, since we are considering only the cohomologies of degree less
or equal to $n-m$, where $n-m$ is the dimension of $V$, in formula \eqref{LOperator}
we have the exact equalities $q=n-m$, $|J|=m$, and therefore $J=(1,\dots,m)$.\\
\indent
First we transform formula \eqref{LOperator} for
${\dis L_q^{\epsilon}\left[\Phi\right](z) }$ with $z\in U^{\epsilon}$
by integrating with respect to variables $\mu_k \in \Delta^{\prime}_J$, and obtain
\begin{multline}\label{LDeterminant}
L_q^{\epsilon}\left[\Phi\right](z)\\
=(-1)^{n-q-1}\frac{n!}{(2\pi i)^{n+1}}\int_{\Gamma^{\epsilon}_J\times\Delta^{\prime}_J}
\Phi(\zeta)\wedge\omega^{\prime}_q\left((1-\mu)\frac{\bar z}
{B^*(\zeta,z)}+\sum_{k=1}^m\mu_k\frac{Q_k(\zeta,z)}
{P_k(\zeta)-P_k(z)}\right)\wedge\omega(\zeta)\\
=C(n,m,d)\int_{\Gamma^{\epsilon}_J}\Phi(\zeta)
\wedge\det\left[\frac{\bar z}{B^*(\zeta,z)}\
\overbrace{\frac{Q_k(\zeta,z)}{P_k(\zeta)-P_k(z)}}^m\
\overbrace{\frac{d{\bar z}}{B^*(\zeta,z)}}^{q=n-m}\right]
\wedge\omega(\zeta).
\end{multline}
\indent
Then, using expression
$$B^*(\zeta,z)=\sum_{j=0}^n{\bar z}_j\cdot\left(\zeta_j-z_j\right)
=-1+\sum_{j=0}^n{\bar z}_j\cdot\zeta_j$$
and its corollary
$$\left(B^*(\zeta,z)\right)^{-q-1}=(-1)^{q+1}
\left(1-\sum_{j=0}^n{\bar z}_j\cdot\zeta_j\right)^{-q-1}$$
in the integral from the right-hand side of \eqref{LDeterminant} we obtain
\begin{multline*}
\int_{\Gamma^{\epsilon}_J}\Phi(\zeta)
\wedge\det\left[\frac{\bar z}{B^*(\zeta,z)}\
\overbrace{\frac{Q_k(\zeta,z)}{P_k(\zeta)-P_k(z)}}^m\
\overbrace{\frac{d{\bar z}}{B^*(\zeta,z)}}^{q}\right]
\wedge\omega(\zeta)\\
=(-1)^{q+1}\lim_{\stackrel{\eta\to 1}{\eta<1}}
\int_{\left\{|\zeta|=1,\left\{|P_k(\zeta)|=\epsilon_k\right\}_{k=1}^m\right\}}
\left(1-\eta\sum_{j=0}^n{\bar z}_j\cdot\zeta_j\right)^{-q-1}\\
\times\Phi(\zeta)\wedge\det\left[{\bar z}\overbrace{\frac{Q_k(\zeta,z)}
{P_k(\zeta)-P_k(z)}}^m\
\overbrace{d{\bar z}}^{q}\right]\wedge\omega(\zeta)
\end{multline*}
\begin{equation}\label{L-Sum-of-Determinants}
=(-1)^{q+1}\lim_{\stackrel{\eta\to 1}{\eta<1}}\sum_{r=0}^{\infty}
c_r\eta^r\cdot\int_{\left\{|\zeta|=1,\left\{|P_k(\zeta)|=\epsilon_k\right\}_{k=1}^m\right\}}
\langle{\bar z}\cdot\zeta\rangle^r\\
\Phi(\zeta)\wedge\det\left[{\bar z}\ \overbrace{\frac{Q_k(\zeta,z)}{P_k(\zeta)-P_k(z)}}^m\
\overbrace{d{\bar z}}^{q}\right]\wedge\omega(\zeta),
\end{equation}
where we denoted $\langle{\bar z}\cdot\zeta\rangle=\sum_{j=0}^n{\bar z}_j\zeta_j$.

\indent
For $\zeta,z\in \*S^{2n+1}(1)$ such that $\left\{|P_k(\zeta)|=\epsilon_k\right\}_{k=1}^m$
and $\left\{|P_k(z)|<\epsilon_k\right\}_{k=1}^m$ we use in the differential form
$$\det\left[{\bar z}\ \overbrace{\frac{Q_k(\zeta,z)}{P_k(\zeta)-P_k(z)}}^{m}\
\overbrace{d{\bar z}}^{q}\right]$$
the following representation with absolutely converging series
\begin{equation}\label{QSeries}
\frac{Q^s_k(\zeta,z)}{P_k(\zeta)-P_k(z)}=\frac{Q^s_k(\zeta,z)}{P_k(\zeta)}
\cdot\left(1-\frac{P_k(z)}{P_k(\zeta)}\right)^{-1}
=\frac{Q^s_k(\zeta,z)}{P_k(\zeta)}\cdot\left(1+\sum_{l=1}^{\infty}
\left(\frac{P_k(z)}{P_k(\zeta)}\right)^l\right)
\end{equation}
and obtain the equality
\begin{multline}\label{SeriesRepresentation}
\int_{\left\{|\zeta|=1,\left\{|P_k(\zeta)|=\epsilon_k\right\}_{k=1}^m\right\}}
\langle{\bar z}\cdot\zeta\rangle^r\Phi(\zeta)
\times\det\left[{\bar z}\
\overbrace{\frac{Q_k(\zeta,z)}{P_k(\zeta)-P_k(z)}}^{m}\
\overbrace{d{\bar z}}^{n-m}\right]\wedge\omega(\zeta)\\
=\sum_{|A|\geq 0}C(A)
\int_{\left\{|\zeta|=1,\left\{|P_k(\zeta)|=\epsilon_k\right\}_{k=1}^m\right\}}
\langle{\bar z}\cdot\zeta\rangle^r\frac{\Phi(\zeta)}{\prod_{k=1}^mP_k(\zeta)}
\cdot\frac{P^A(z)}{P^A(\zeta)}
\wedge\det\left[{\bar z}\ \overbrace{Q_k(\zeta,z)}^{n-q}\
\overbrace{d{\bar z}}^{n-m}\right]\wedge\omega(\zeta),
\end{multline}
where $A=(a_1,\dots,a_m)$ is a multiindex,
$$P^A(\zeta)=P_1^{a_1}(\zeta)\cdots P_m^{a_m}(\zeta),$$
and $|A|=a_1+\cdots+a_m$.\\
\indent
Using Theorem 1.7.6(2) from \cite{CH} (see also \cite{HP3} Prop. 2.3) we obtain that the residual
currents defined by the terms in the right-hand side of \eqref{SeriesRepresentation} with $|A|\geq 1$
are zero-currents from the point of view of \eqref{preL-Operator}, and therefore
we can simplify the expression for
$L_q^{\epsilon}\left[\Phi\right]$ as follows
\begin{multline}\label{NoSeriesRepresentation}
L_q^{\epsilon}\left[\Phi\right](z)
=C(n,m,d)\lim_{\stackrel{\eta\to 1}{\eta<1}}\sum_{r=0}^{\infty}
c_r\eta^r\cdot\int_{\left\{|\zeta|=1,\left\{|P_k(\zeta)=\epsilon_k\right\}_{k=1}^m\right\}}
\langle{\bar z}\cdot\zeta\rangle^r\\
\times\Phi(\zeta)\wedge\det\left[{\bar z}\
\overbrace{\frac{Q_k(\zeta,z)}{P_k(\zeta)-P_k(z)}}^{n-q}\
\overbrace{d{\bar z}}^{q}\right]\wedge\omega(\zeta)\\
=C(n,m,d)\lim_{\stackrel{\eta\to 1}{\eta<1}}\sum_{r=0}^{\infty}
\int_{\left\{|\zeta|=1,\left\{|P_k(\zeta)=\epsilon_k\right\}_{k=1}^m\right\}}
c_r\eta^r\langle{\bar z}\cdot\zeta\rangle^r\frac{\Phi(\zeta)}{\prod_{k=1}^mP_k(\zeta)}\\
\bigwedge\det\left[{\bar z}\ \overbrace{Q_k(\zeta,z)}^{n-q}\
\overbrace{d{\bar z}}^{q}\right]\wedge\omega(\zeta).
\end{multline}

\indent
Before continuing with further transformation of
\eqref{NoSeriesRepresentation} we prove a lemma, in which we slightly modify
the result of Coleff and Herrera from \cite{CH} to obtain the existence of residual limits
over deformed admissible tubes for reduced complete intersections.

\begin{lemma}\label{LimitsExistence}
Let $\left\{F_1,\dots,F_m\right\}$ be polynomials on $\C^n$, let
\begin{equation}
V=\left\{\zeta\in\C^n: F_1(\zeta)=\cdots=F_m(\zeta)=0\right\}
\end{equation}
be a reduced complete intersection subvariety, and let $g$ be a holomorphic function $g$ satisfying:
\begin{itemize}
\item[(i)]
$V^{\prime}=\left\{\zeta: F_1(\zeta)=\cdots=F_m(\zeta)=g(\zeta)=0\right\}$
is a complete intersection,\vspace{0.1in}
\item[(ii)]
for any $z\in V\setminus V^{\prime}$ there exists a neighborhood $W_z$, such that
$\left(V\cap W_z\right)\setminus V^{\prime}$ is a submanifold in $W_z$.
\end{itemize}

\indent
Then for an arbitrary differential form $\Phi(\zeta,u)\in {\cal E}_c^{(n,n-m)}\left(\C^n\right)$
real analytic with respect to parameters $u_1,\dots, u_s$, and a collection of real-valued functions
$\left\{\chi_k(\zeta)\right\}_{k=1}^m\in {\cal E}_c(\C^n)$, such
that $\chi_k(\zeta)\geq 1$ for $|\zeta|<1$, the limit along an admissible path
$\left\{\epsilon_k(t)\right\}_{k=1}^m$ defined in \eqref{admissible}
\begin{multline}\label{LimitExistence}
\lim_{t\to 0}\int_{\left\{|F_k(\zeta)|\cdot\chi_k(\zeta)=\epsilon_k(t)\right\}_{k=1}^m}
\frac{\Phi(\zeta,u)}{\prod_{k=1}^m F_k(\zeta)}\\
\stackrel{\rm def}{=}\lim_{\eta\to 0}\lim_{t\to 0}
\int_{\left\{|g(\zeta)|>\eta,\
\left\{|F_k(\zeta)|\cdot\chi_k(\zeta)=\epsilon_k(t)\right\}_{k=1}^m\right\}}
\frac{\Phi(\zeta,u)}{\prod_{k=1}^m F_k(\zeta)}\\
=\lim_{\eta\to 0}\lim_{t\to 0}
\int_{\left\{|g(\zeta)|>\eta,\
\left\{|F_k(\zeta)|=\epsilon_k(t)\right\}_{k=1}^m\right\}}
\frac{\Phi(\zeta,u)}{\prod_{k=1}^m F_k(\zeta)}
\end{multline}
exists and is real analytic with respect to parameters $u_1,\dots, u_s$.
\end{lemma}
\begin{proof}
We assume that the analytic set $V$ is a subset of a polydisk
${\cal P}^n=\left\{|\zeta_i|<1,\ i=1,\dots,n\right\}$, such that the restriction of the projection
$$\pi:\ {\cal P}^n\to {\cal P}^{n-m},$$
defined by the formula $\pi(\zeta_1,\dots,\zeta_n)=(\zeta_{m+1},\dots,\zeta_n)$,
to $V\cap{\cal P}$ is a finite proper covering, and the holomorphic function $g$ on ${\cal P}^n$ is
such that $\dim\big\{V\cap\{g(\zeta)=0\}\big\}= n-m-1$.\\
\indent
For a point $z\in V$, such that $|g(z)|>\eta$ we consider the nonholomorphic complex coordinates
$$w_1(\zeta)=F_1(\zeta)\cdot\chi_1(\zeta),\dots,
w_m(\zeta)=F_m(\zeta)\cdot\chi_m(\zeta),\zeta_{m+1},\dots,\zeta_n$$
in a small enough neighborhood of the point $z$. Then for the $(m,0)$-form
$$\Phi(\zeta,z,u)=\phi(\zeta,z,u)\bigwedge_{j=1}^md\zeta_j
=\left(\bigwedge_{j=m+1}^nd\zeta_j
\bigwedge_{j=m+1}^nd{\bar\zeta}_j\right)\interior\ \Phi(\zeta,u)
\Bigg|_{\left\{\zeta_j=\zeta_j(z)\right\}_{j=m+1}^n}$$
we have
\begin{multline*}
\lim_{t\to 0}\int_{\left\{|F_k(\zeta)|\cdot\chi_k(\zeta)=\epsilon_k(t)\right\}_{k=1}^m
\left\{\zeta_j=\zeta_j(z)\right\}_{j=m+1}^n}
\frac{\Phi(\zeta,z,u)}{\prod_{k=1}^m F_k(\zeta)}\\
=\lim_{t\to 0}\int_{\left\{|w_k(\zeta)|=\epsilon_k(t)\right\}_{k=1}^m
\left\{\zeta_j=\zeta_j(z)\right\}_{j=m+1}^n}
\frac{\Psi(\zeta,z,u)}{\prod_{k=1}^m w_k(\zeta)},
\end{multline*}
where
$$\Psi(\zeta,z,u)=\psi(\zeta,z,u)\bigwedge_{j=1}^md\zeta_j
=\Phi(\zeta,z,u)\cdot\prod_{k=1}^m\chi_k(\zeta).$$
\indent
Using equalities
$$\Psi(\zeta,z,u)\Bigg|_{\left\{\zeta_j=\zeta_j(z)\right\}_{j=m+1}^n}
=\phi(\zeta,z,u)\cdot\prod_{k=1}^m\chi_k(\zeta)
\cdot{\det}^{-1}\left[\frac{\partial F_k}{\partial\zeta_l}\right]
\bigwedge_{k=1}^m dF_k,$$
\begin{multline*}
\frac{\partial w_j}{\partial\zeta_l}\Bigg|_{\left\{F_k(\zeta)=0\right\}_{k=1}^m
\left\{\zeta_j=\zeta_j(z)\right\}_{j=m+1}^n}
=\frac{\partial}{\partial\zeta_l}\left[F_j\cdot\chi_j(\zeta)\right]
\Bigg|_{\left\{F_k(\zeta)=0\right\}_{k=1}^m\left\{\zeta_j=\zeta_j(z)\right\}_{j=m+1}^n}\\
=\frac{\partial F_j}{\partial\zeta_l}(\zeta)
\cdot\chi_j(\zeta)\Bigg|_{\left\{F_k(\zeta)=0\right\}_{k=1}^m
\left\{\zeta_j=\zeta_j(z)\right\}_{j=m+1}^n}\ \text{for}\ j=1,\dots,m,
\end{multline*}
and the corollary of the second one
$$\prod_{k=1}^m\chi_k(\zeta)\bigwedge_{k=1}^m dF_k(\zeta)
\Bigg|_{\left\{F_k(\zeta)=0\right\}_{k=1}^m\left\{\zeta_j=\zeta_j(z)\right\}_{j=m+1}^n}
=\bigwedge_{k=1}^m dw_k(\zeta)
\Bigg|_{\left\{F_k(\zeta)=0\right\}_{k=1}^m\left\{\zeta_j=\zeta_j(z)\right\}_{j=m+1}^n},$$
we obtain for $z$ with $|g(z)|>\delta$ the equality
\begin{multline*}
\lim_{t\to 0}\int_{\left\{|w_k(\zeta)|=\epsilon_k(t)\right\}_{k=1}^m
\left\{\zeta_j=\zeta_j(z)\right\}_{j=m+1}^n}
\frac{\Psi(\zeta,z,u)}{\prod_{k=1}^m w_k(\zeta)}\\
=\lim_{t\to 0}\int_{\left\{|w_k(\zeta)|=\epsilon_k(t)\right\}_{k=1}^m
\left\{\zeta_j=\zeta_j(z)\right\}_{j=m+1}^n}
\phi(\zeta,z,u)\cdot\prod_{k=1}^m\chi_k(\zeta)
\cdot{\det}^{-1}\left[\frac{\partial F_k}{\partial\zeta_l}\right]
\frac{\bigwedge_{k=1}^m dF_k(\zeta)}{\prod_{k=1}^m w_k(\zeta)}\\
=\lim_{t\to 0}\int_{\left\{|w_k(\zeta)|=\epsilon_k(t)\right\}_{k=1}^m
\left\{\zeta_j=\zeta_j(z)\right\}_{j=m+1}^n}
\phi(\zeta,z,u)\cdot{\det}^{-1}\left[\frac{\partial F_k}{\partial\zeta_l}\right]
\frac{\bigwedge_{k=1}^m dw_k(\zeta)}{\prod_{k=1}^m w_k(\zeta)}\\
=(2\pi i)^k\phi(\zeta(z),z,u)\cdot{\det}^{-1}\left[\frac{\partial F_k}{\partial\zeta_l}\right]
\left(\zeta(z)\right)\\
=\lim_{t\to 0}\int_{\left\{|F_k(\zeta)|=\epsilon_k(t)\right\}_{k=1}^m
\left\{\zeta_j=\zeta_j(z)\right\}_{j=m+1}^n}
\Phi(\zeta,z,u)\cdot{\det}^{-1}\left[\frac{\partial F_k}{\partial\zeta_l}\right]
\frac{\bigwedge_{k=1}^m dF_k(\zeta)}{\prod_{k=1}^m F_k(\zeta)}.\\
\end{multline*}
\indent
From the last equality we obtain the equality
\begin{equation*}
\lim_{t\to 0}\int_{\left\{|F_k(\zeta)|\cdot\chi_k(\zeta)=\epsilon_k(t)\right\}_{k=1}^m
\left\{\zeta_j=\zeta_j(z)\right\}_{j=m+1}^n}
\frac{\Phi(\zeta,z,u)}{\prod_{k=1}^m F_k(\zeta)}
=\mbox{res}_{\{{\bf F},\pi\}}\left(\Phi,z\right),
\end{equation*}
which in combination with equality
\begin{equation}\label{FiberedEquality}
\lim_{t\to 0}\int_{\left\{|F_k(\zeta)|=\epsilon_k(t)\right\}_{k=1}^m}
\frac{\Phi(\zeta,u)}{\prod_{k=1}^m F_k(\zeta)}
=\lim_{\eta\to 0}\int_{V\cap\{|g(z)|>\eta\}}\mbox{res}_{\{{\bf F},\pi\}}
\left(\Phi,z\right),
\end{equation}
from Theorem 1.8.3 in \cite{CH} (see also \cite{HP4} Prop. 2.2) and existence of the limit in the left-hand side of
\eqref{FiberedEquality}, following from Theorem 1.7.2 in \cite{CH}, implies the existence
of the limit in the right-hand side of \eqref{LimitExistence}.\\
\indent
To prove the real analyticity of the limit in the right-hand side of \eqref{LimitExistence}
with respect to real variables $u_1,\dots, u_s$ we represent those variables
in terms of complex variables
$$u_r=1/2\left(w_r+{\bar w}_r\right).$$
Then the resulting form can be considered as a restriction of a form analytically depending
on $2s$ complex variables $\left\{w_1,\dots,w_s, v_1,\dots,v_s\right\}$ obtained after
substitution $v_r={\bar w}_r$. Then from Lemma 2.4 in \cite{HP4} we obtain an analytic
dependence of the residual integral on $\left\{w_1,\dots,w_s, v_1,\dots,v_s\right\}$,
and, as a corollary, its real analytic dependence on the original parameters
$u_1,\dots, u_s$.
\end{proof}

\indent
In the next lemma using Lemma~\ref{LimitsExistence} we prove the existence of residual limits
for the integrals on a sphere in $\C^{n+1}$, which are present in
formula \eqref{NoSeriesRepresentation}.

\begin{lemma}\label{ResidueReduction} Let $V\subset \C\P^n$ be a reduced complete intersection subvariety as in \eqref{Variety} satisfying \eqref{gFunctions},
let $U\supset V$ be an open neighborhood
of $V$ in $\C\P^n$, and let $\Phi\in {\cal E}_c^{(0,n-m)}(U\cap U_{\alpha})$
be a differential form of homogeneity zero on $U\cap U_{\alpha}$ for
some $\alpha\in (0,\dots,n)$.\\
\indent
Then formula
\begin{equation}\label{SphericalIntegral}
\lim_{t\to 0}\int_{\left\{|\zeta|=\tau,\left\{|P_k(\zeta)|=\epsilon_k(t)\right\}_{k=1}^m\right\}}
\langle{\bar z}\cdot\zeta\rangle^r\cdot\frac{\Phi(\zeta)}{\prod_{k=1}^mP_k(\zeta)}
\wedge\det\left[{\bar z}\ \overbrace{Q(\zeta,z)}^{m}\
\overbrace{d{\bar z}}^{n-m}\right]\wedge\omega(\zeta),
\end{equation}
where $\left\{\epsilon_k(t)\right\}_{k=1}^m$ is an admissible path,
defines a differential form of homogeneity zero on $U$, real analytic with respect to $z$.\\
\indent
If $\Phi(\zeta)=\sum_{k=1}^m F^{(\alpha)}_k(\zeta)\Omega_k(\zeta)$ with
$\Omega_k\in {\cal E}_c^{(0,n-m)}(U\cap U_{\alpha})$,
then the limit above is equal to zero.
\end{lemma}
\begin{proof} Without loss of generality we may assume that $\alpha=0$ in \eqref{SphericalIntegral}.
We transform the integral in this formula as follows
\begin{multline*}
\int_{\left\{|\zeta|=\tau,\left\{|P_k(\zeta)|=\epsilon_k(t)\right\}_{k=1}^m\right\}}
\langle{\bar z}\cdot\zeta\rangle^r\cdot\frac{\Phi(\zeta)}{\prod_{k=1}^mP_k(\zeta)}
\wedge\det\left[{\bar z}\ \overbrace{Q(\zeta,z)}^{m}\
\overbrace{d{\bar z}}^{n-m}\right]\wedge\omega(\zeta)\\
=\int_{\left\{|\zeta|=\tau,\left\{|P_k(\zeta)|=\epsilon_k(t)\right\}_{k=1}^m\right\}}
\frac{\langle{\bar z}\cdot\zeta\rangle}{\tau^2}^r\cdot\frac{\Phi(\zeta)}{\prod_{k=1}^mP_k(\zeta)}
\wedge\det\left[{\bar z}\ \overbrace{Q(\zeta,z)}^{m}\ \overbrace{d{\bar z}}^{n-m}\right]\\
\wedge\left(\sum_{i=0}^n{\bar\zeta}_id\zeta_i\right)\wedge\omega^{\prime}(\zeta),
\end{multline*}
where $\omega^{\prime}(\zeta)=\sum_{i=0}^n(-1)^{i}\zeta_id\zeta_0
\wedge\stackrel{\stackrel{i}{\vee}}{\cdots}\wedge d\zeta_n$.\\
\indent
Then, using the nonhomogeneous coordinates
\begin{equation}\label{Nonhomogeneous}
\zeta_0,w_1=\zeta_1/\zeta_0,\dots,w_n=\zeta_n/\zeta_0
\end{equation}
and equality
\begin{equation}\label{r0Dependence}
1+\sum_{i=1}^nw_i\cdot{\bar w}_i=\frac{\tau^2}{\zeta_0\cdot{\bar\zeta}_0}
\end{equation}
on the sphere $\*S^{2n+1}(\tau)$ of radius $\tau$ in $\C^{n+1}$ we represent the form
$\sum_{i=0}^n{\bar\zeta}_id\zeta_i$ in
$${\widetilde U}_0(\tau)=\left\{\zeta\in\C^{n+1}:\ |\zeta|=\tau, \zeta_0\neq 0\right\}$$
as
\begin{multline}\label{dModuleRepresentation}
\sum_{i=0}^n{\bar\zeta}_id\zeta_i={\bar\zeta}_0d\zeta_0
+\sum_{i=1}^n{\bar\zeta}_0\cdot{\bar w}_i\left(\zeta_0dw_i+w_id\zeta_0\right)\\
={\bar\zeta}_0\left(1+\sum_{i=1}^n{\bar w}_i\cdot w_i\right)d\zeta_0
+\zeta_0\cdot {\bar\zeta}_0\left(\sum_{i=1}^n{\bar w}_idw_i\right)
=\frac{\tau^2}{\zeta_0}d\zeta_0+\zeta_0\cdot {\bar\zeta}_0\left(\sum_{i=1}^n{\bar w}_idw_i\right).
\end{multline}
\indent
For the form $\omega^{\prime}(\zeta)$ using equalities $d\zeta_i=\zeta_0dw_i+w_id\zeta_0$
for $i=1,\dots,n$ we obtain
\begin{multline}\label{OmegaRepresentation}
\omega^{\prime}(\zeta)=\sum_{i=0}^n(-1)^{i}\zeta_id\zeta_0\wedge\stackrel{\stackrel{i}{\vee}}
{\cdots}\wedge d\zeta_n=\zeta_0\bigwedge_{j=1}^n\left(\zeta_0dw_j+w_jd\zeta_0\right)\\
-\zeta_0w_1d\zeta_0\wedge\left(\zeta_0dw_2+w_2d\zeta_0\right)\wedge\cdots
\wedge\left(\zeta_0dw_n+w_nd\zeta_0\right)\\
+\zeta_0w_2d\zeta_0\wedge\left(\zeta_0dw_1+w_1d\zeta_0\right)\wedge
\stackrel{\stackrel{2}{\vee}}{\cdots}
\wedge\left(\zeta_0dw_n+w_nd\zeta_0\right)\\
+\cdots\\
+(-1)^n\zeta_0w_nd\zeta_0\wedge\left(\zeta_0dw_1+w_1d\zeta_0\right)\wedge\cdots
\wedge\left(\zeta_0dw_{n-1}+w_{n-1}d\zeta_0\right)\\
=\zeta_0\bigwedge_{j=1}^n\left(\zeta_0dw_j+w_jd\zeta_0\right)
+\zeta_0^n\sum_{j=1}(-1)^jw_jd\zeta_0\wedge dw_1
\wedge\stackrel{\stackrel{j}{\vee}}{\cdots}\wedge dw_n\\
=\zeta_0^{n+1}dw_1\wedge\cdots\wedge dw_n.
\end{multline}
\indent
Using formulas \eqref{dModuleRepresentation} and \eqref{OmegaRepresentation},
we obtain the equality
\begin{multline*}
\int_{\left\{|\zeta|=\tau,\left\{|P_k(\zeta)|=\epsilon_k(t)\right\}_{k=1}^m\right\}}
\frac{\langle{\bar z}\cdot\zeta\rangle}{\tau^2}^r\cdot\frac{\Phi(\zeta)}{\prod_{k=1}^mP_k(\zeta)}
\wedge\det\left[{\bar z}\ \overbrace{Q(\zeta,z)}^{m}\ \overbrace{d{\bar z}}^{n-m}\right]
\wedge\left(\sum_{i=0}^n{\bar\zeta}_id\zeta_i\right)\wedge\omega^{\prime}(\zeta)\\
=\int_{\left\{|\zeta|=\tau,\left\{|P_k(\zeta)|=\epsilon_k(t)\right\}_{k=1}^m\right\}}
\frac{\langle{\bar z}\cdot\zeta\rangle}{\tau^2}^r\cdot\frac{\Phi(\zeta)}{\prod_{k=1}^mP_k(\zeta)}
\wedge\det\left[{\bar z}\ \overbrace{Q(\zeta,z)}^{m}\ \overbrace{d{\bar z}}^{n-m}\right]\\
\wedge\left(\frac{\tau^2}{\zeta_0}d\zeta_0
+\zeta_0\cdot {\bar\zeta}_0\left(\sum_{i=1}^n{\bar w}_idw_i\right)\right)
\wedge\zeta_0^{n+1}\bigwedge_{j=1}^n dw_j\\
=\int_{\left\{|\zeta|=\tau,\left\{|P_k(\zeta)|=\epsilon_k(t)\right\}_{k=1}^m\right\}}
\left({\bar z}_0+\sum_{j=1}^n{\bar z}_j\cdot w_j\right)^r\cdot
\frac{\Phi(\zeta)}{\prod_{k=1}^mP_k(\zeta)}\\
\wedge\det\left[{\bar z}\ \overbrace{Q(\zeta,z)}^{m}\ \overbrace{d{\bar z}}^{n-m}\right]
\wedge \left(\zeta_0^{n+r}d\zeta_0\right)\bigwedge_{j=1}^n dw_j.
\end{multline*}
\indent
Then, using the nonhomogeneous polynomials
\begin{equation}\label{Polynomials}
F_k(w)=F^{(0)}_k(w)=P_k(\zeta)/\zeta_0^{\deg P_k}
\end{equation}
and denoting $\chi(w)=\left(1+\sum_{i=1}^n w_i{\bar w}_i\right)^{-\frac{1}{2}}$,
and $\chi_k(w)=\chi(w)^{\deg P_k}$, so that
$$\left|P_k(\zeta)\right|=|F_k(w)|\cdot|\zeta_0|^{\deg P_k}=|F_k(w)|\cdot\chi_k(w)\
\text{on}\ \*S^{2n+1}(1)$$
we obtain the equality
\begin{multline}\label{LocalIntegralEquality}
\int_{\left\{|\zeta|=\tau,\left\{|P_k(\zeta)|=\epsilon_k(t)\right\}_{k=1}^m\right\}}
\left({\bar z}_0+\sum_{j=1}^n{\bar z}_j\cdot w_j\right)^r\cdot
\frac{\Phi(\zeta)}{\prod_{k=1}^mP_k(\zeta)}\\
\wedge\det\left[{\bar z}\ \overbrace{Q(\zeta,z)}^{m}\ \overbrace{d{\bar z}}^{n-m}\right]
\wedge \left(\zeta_0^{n+r}d\zeta_0\right)\bigwedge_{j=1}^n dw_j\\
=\int_{\left\{|\zeta|=\tau,\left\{|F_k(w)|\cdot\chi_k(w)=\epsilon_k(t)\right\}_{k=1}^m\right\}}
\left({\bar z}_0+\sum_{j=1}^n{\bar z}_j\cdot w_j\right)^r
\left(\zeta_0^{n+r-\sum_{k=1}^m \deg P_k}d\zeta_0\right)\\
\times\frac{\Phi(\zeta)}{\prod_{k=1}^mF_k(w)}
\wedge\det\left[{\bar z}\ \overbrace{Q(\zeta,z)}^{m}\ \overbrace{d{\bar z}}^{n-m}\right]
\bigwedge_{j=1}^n dw_j\\
=\int_{\left\{|\zeta|=\tau,\left\{|F_k(w)|\cdot\chi_k(w)=\epsilon_k(t)\right\}_{k=1}^m\right\}}
\left({\bar z}_0+\sum_{j=1}^n{\bar z}_j\cdot w_j\right)^r
\left(\zeta_0^{n+r-\sum_{k=1}^m \deg P_k}d\zeta_0\right)\\
\times\frac{\Phi(w)}{\prod_{k=1}^m F_k(w)}
\wedge\det\left[{\bar z}\ \overbrace{Q(\zeta,z)}^{m}\ \overbrace{d{\bar z}}^{n-m}\right]
\bigwedge_{j=1}^n dw_j\\
=i\int_0^{2\pi}e^{i\left(n+r-\sum_{k=1}^m \deg P_k+1\right)\phi_0}d\phi_0
\int_{\left\{w\in U_0,\left\{|F_k(w)|\cdot\chi_k(w)=\epsilon_k(t)\right\}_{k=1}^m\right\}}
\left({\bar z}_0+\sum_{j=1}^n{\bar z}_j\cdot w_j\right)^r\\
\times\rho_0(w)^{n+r-\sum_{k=1}^m \deg P_k+1}
\cdot\frac{\Phi(w)}{\prod_{k=1}^m F_k(w)}
\wedge\det\left[{\bar z}\ \overbrace{Q(e^{i\phi_0},w,z)}^{m}\
\overbrace{d{\bar z}}^{n-m}\right]\bigwedge_{j=1}^n dw_j,
\end{multline}
where we used notation $\zeta_0=\rho_0(w)\cdot e^{i\phi_0}$ with
$$\rho_0(w)=\frac{\tau}{\sqrt{1+\sum_{i=1}^n|w_i|^2}}$$
depending on $\left\{w_i,{\bar w}_i\right\}_{i=1}^n$ on the sphere $\*S^{2n+1}(\tau)$
according to formula \eqref{r0Dependence}.\\
\indent
Applying Lemma~\ref{LimitsExistence} to the interior integral in the right-hand side of
\eqref{LocalIntegralEquality} we obtain the existence of the limit in \eqref{SphericalIntegral} and
its real analytic dependence on $z$. Applying Theorem 1.7.6(2) from \cite{CH} (see also \cite{HP3} Prop. 2.3)
to the interior integral in the right-hand side of \eqref{LocalIntegralEquality} we obtain
that the limit in \eqref{SphericalIntegral} is equal to zero if
$\Phi(\zeta)=\sum_{k=1}^m F^{(\alpha)}_k(\zeta)\Omega_k(\zeta)$ with
$\Omega_k\in {\cal E}_c^{(0,n-m)}(U\cap U_{\alpha})$.
\end{proof}

\indent
We further simplify the right-hand side of \eqref{NoSeriesRepresentation} using the
following lemma.

\begin{lemma}\label{ZeroIntegrals} Let $\phi$ be a $\bar\partial$-closed residual current defined
by a collection of forms $\left\{\Phi^{(0,n-m)}_{\alpha}\right\}_{\alpha=0}^n$
of homogeneity zero on a neighborhood $U$ of the reduced subvariety
$$V=\Big\{\zeta\in \C\P^n:\ P_1(\zeta)=\cdots=P_m(\zeta)=0\Big\}$$
satisfying \eqref{ResidualCurrent} and \eqref{Closed},
and let $\Phi(\zeta)=\sum_{\alpha=0}^n\vartheta_{\alpha}(\zeta)\Phi_{\alpha}(\zeta)$
be a differential form of homogeneity zero on $U$.\\
\indent
Then for an arbitrary
$\gamma\in {\cal E}_c^{(n,0)}\left(V,{\cal L}\right)$
the equality
\begin{multline}\label{ZeroIntegral}
\lim_{\tau\to 0}\int_{T^{\delta(\tau)}_{\beta}}
\wedge\frac{\gamma(z)}{\prod_{k=1}^m F^{(\beta)}_k(z)}
\bigwedge\left(\lim_{t\to 0}\int_{\left\{|\zeta|=1,\left\{|P_k(\zeta)|=\epsilon_k(t)\right\}_{k=1}^m\right\}}
\langle{\bar z}\cdot\zeta\rangle^r\cdot\frac{\Phi(\zeta)}{\prod_{k=1}^mP_k(\zeta)}\right.\\
\left.\wedge\det\left[{\bar z}\ \overbrace{Q(\zeta,z)}^{m}\
\overbrace{d{\bar z}}^{n-m}\right]\wedge\omega(\zeta)\right)=0
\end{multline}
holds unless
\begin{equation}\label{IndexCondition}
r\leq\sum_{k=1}^m\deg P_k-n-1.
\end{equation}
\end{lemma}
\begin{proof} We notice that for all values of $a>0$
and $\epsilon=\left(\epsilon_1,\dots,\epsilon_k\right)$ the sets
$$S( a)=\left\{\zeta\in \*S^{2n+1}(a):
\left\{ \left|P_k(\zeta)\right|=\epsilon_k\cdot a^{\deg P_k}\right\}_{k=1}^m\right\}$$
are real analytic subvarieties of $\*S^{2n+1}(a)$ of real dimension $2n+1-m$
satisfying
\begin{equation}\label{VolumeEstimate}
c\cdot a^{2n+1-m}\cdot\text{Volume}\left(S(1)\right)
<\text{Volume}_{2n+1-m}\left(S(a)\right)
<C\cdot a^{2n+1-m}\cdot\text{Volume}\left(S(1)\right).
\end{equation}
\indent
We denote
$$\Phi_{\alpha}(\zeta,z)=\langle{\bar z}\cdot\zeta\rangle^r\cdot
\frac{\Phi_{\alpha}(\zeta)}{\prod_{k=1}^mP_k(\zeta)}
\det\left[{\bar z}\ \overbrace{Q(\zeta,z)}^{m}\ \overbrace{d{\bar z}}^{n-m}\right]
\wedge\omega(\zeta),$$
and apply the Stokes' formula to the differential form
\begin{equation}\label{betaform}
\beta(\zeta,z)=\sum_{\alpha=1}^N\vartheta_{\alpha}(\zeta)\Phi_{\alpha}(\zeta,z)
=\sum_{\alpha=1}^N\langle{\bar z}\cdot\zeta\rangle^r
\cdot\frac{\vartheta_{\alpha}(\zeta)\Phi_{\alpha}(\zeta)}
{\prod_{k=1}^mP_k(\zeta)}
\det\left[{\bar z}\ \overbrace{Q(\zeta,z)}^{m}\ \overbrace{d{\bar z}}^{n-m}\right]\wedge\omega(\zeta)
\end{equation}
on the variety
$$\left\{\zeta\in \C^{n+1}:
\left\{ \left|P_k(\zeta)\right|=\epsilon_k\cdot|\zeta|^{\deg P_k}\right\}_{k=1}^m,\
a<|\zeta|<1\right\}$$
with the boundary
\begin{equation*}
\Big\{\zeta:|\zeta|=a,\
\left\{|P_k(\zeta)|=\epsilon_k\cdot a^{\deg P_k}\right\}_{k=1}^m\Big\}
\bigcup\Big\{\zeta:|\zeta|=1,\ \left\{|P_k(\zeta)|=\epsilon_k\right\}_{k=1}^m\Big\}.
\end{equation*}
Then using equality \eqref{Closed}
we obtain the equality
\begin{multline}\label{epsilonStokes}
\int_{\left\{|\zeta|=1, \left\{|P_k(\zeta)|=\epsilon_k(t)\right\}_{k=1}^m\right\}}\beta(\zeta,z)
-\int_{\left\{|\zeta|=a,\
\left\{|P_k(\zeta)|=\epsilon_k(t)\cdot a^{\deg P_k}\right\}_{k=1}^m\right\}}\beta(\zeta,z)\\
=\sum_{\alpha=1}^N\sum_{j=1}^m\int_{a}^1d\tau
\int_{\left\{|\zeta|=\tau, \left\{|P_k(\zeta)|=\epsilon_k(t)\cdot\tau^{\deg P_k}\right\}_{k=1}^m\right\}}
F^{(\alpha)}_j(\zeta)\cdot\beta_j^{(\alpha)}(\zeta,z)\\
+\sum_{\alpha=1}^N\int_{a}^1d\tau
\int_{\left\{|\zeta|=\tau, \left\{|P_k(\zeta)|=\epsilon_k(t)\cdot\tau^{\deg P_k}\right\}_{k=1}^m\right\}}
\bar\partial\vartheta_{\alpha}(\zeta)\wedge\Big(d|\zeta|\interior\ \Phi_{\alpha}(\zeta,z)\Big)
\end{multline}
for arbitrary $t$ and $0<a<1$.\\
\indent
Using estimate \eqref{VolumeEstimate} and the homogeneity property
\begin{equation}\label{PhiHomogeneity}
\Phi_{\beta}(t\cdot\zeta)={\bar t}^{-(n-m)}\cdot\Phi_{\beta}(\zeta)
\end{equation}
of the coefficients of $\Phi^{(0,n-m)}$ from Proposition 1.1 in \cite{HP1} we obtain that if
\begin{equation}\label{IndexPositive}
r+n+1-\sum_{k=1}^m\deg P_k >0,
\end{equation}
then
\begin{equation*}
\left|\int_{\left\{|\zeta|= a,\
\left\{|P_k(\zeta)|=\epsilon_k(t)\cdot a^{\deg P_k}\right\}_{k=1}^m\right\}}
\beta(\zeta,z)\right|
<C^{\prime}(\epsilon)\cdot a^{r+2n+1-m-(n-m)-\sum_{k=1}^m\deg P_k}\longrightarrow 0
\end{equation*}
as $t$ is fixed and $a\to 0$.\\
\indent
For the first sum of integrals in the right-hand side of \eqref{epsilonStokes} we have
\begin{multline}\label{AreaIntegralOne}
\Bigg|\int_{ a}^1d\tau\int_{\left\{|\zeta|=\tau,\
\left\{|P_k(\zeta)|=\epsilon_k(t)\cdot\tau^{\deg P_k}\right\}_{k=1}^m\right\}}
F^{(\alpha)}_j(\zeta)\cdot\beta_j^{(\alpha)}(\zeta,z)\Bigg|\\
<C\int_{ a}^1d\tau\cdot\tau^{r+2n+1-m-(n-m)-\sum_{k=1}^m\deg P_k}\\
\times\int_{\left\{|\zeta|=\tau,\
\left\{|P_k(\zeta)|=\epsilon_k(t)\cdot\tau^{\deg P_k}\right\}_{k=1}^m\right\}}
F^{(\alpha)}_j(\zeta)\cdot\left(d|\zeta|\interior\beta_j^{(\alpha)}(\zeta,z)\right)\\
<C^{\prime}(\epsilon)\frac{\left(1- a^{r+n+2-\sum_{k=1}^m\deg P_k}\right)}
{r+n+2-\sum_{k=1}^m\deg P_k}\longrightarrow 0
\end{multline}
as $t\to 0$, since $a<1$, condition \eqref{IndexPositive} is satisfied,
and $C^{\prime}(\epsilon)\to 0$ as $t\to 0$ by Lemma~\ref{ResidueReduction}.

\indent
For the second sum of integrals in the right-hand side of \eqref{epsilonStokes}
using equality $\sum_{\alpha}\bar\partial\vartheta_{\alpha}=0$, equality \eqref{ResidualCurrent}
for residual currents of homogeneity zero, and Lemma~\ref{ResidueReduction} we obtain
as in \eqref{AreaIntegralOne} that the limit of this sum is also zero as $t\to 0$.\\
\indent
This completes the proof of the Lemma.
\end{proof}

\indent
Combining the results of Lemmas~\ref{ResidueReduction} and \ref{ZeroIntegrals}
with formula \eqref{LDeterminant} we obtain the following
\begin{proposition}\label{HodgeProjector} Let $V\subset \C\P^n$ be a reduced
complete intersection subvariety as in \eqref{Variety} satisfying conditions
\eqref{gFunctions}, let  $\phi$ be a $\bar\partial$-closed residual current defined
by a collection of forms $\left\{\Phi^{(0,n-m)}_{\alpha}\right\}_{\alpha=0}^n$
of homogeneity zero on a neighborhood $U$ of $V$,
and let $\Phi(\zeta)=\sum_{\alpha=0}^n\vartheta_{\alpha}(\zeta)\Phi_{\alpha}(\zeta)$.\\
\indent
Then for an admissible path $\epsilon(t)$ and operator $L_{n-m}^{\epsilon(t)}$
the following equality is satisfied
\begin{multline}\label{LResidue}
\lim_{t\to 0}L_{n-m}^{\epsilon(t)}\left[\Phi\right](z)
=\sum_{0\leq r\leq d-n-1}C(n,m,d,r)\lim_{t\to 0}\int_{\left\{|\zeta|=1,\
\left\{|P_k(\zeta)|=\epsilon_k(t)\right\}_{k=1}^m\right\}}
\langle{\bar z}\cdot\zeta\rangle^r\\
\times\frac{\Phi(\zeta)}{\prod_{k=1}^mP_k(\zeta)}
\det\left[{\bar z}\ \overbrace{Q(\zeta,z)}^{m}\
\overbrace{d{\bar z}}^{n-m}\right]\wedge\omega(\zeta)\\
=\sum_{0\leq r\leq d-n-1}C(n,m,d,r)\left(2\pi\right)^mi^{m+1}
\sum_{\alpha=0}^n\int_0^{2\pi}d\phi_{\alpha}\\
\times\text{Res}_V\left\{\left\langle{\bar z}\cdot w^{(\alpha)}
\right\rangle^r\Phi_{\alpha}(\zeta)
\wedge\det\left[{\bar z}\ \overbrace{Q(e^{i\phi_{\alpha}},w^{(\alpha)},z)}^{m}\
\overbrace{d{\bar z}}^{n-m}\right]\bigwedge_{j=1}^n dw^{(\alpha)}_j\right\},
\end{multline}
where $d=\sum_{k=1}^m\deg P_k$,
$\langle{\bar z}\cdot w^{(\alpha)}\rangle={\bar z}_0
+\sum_{j=1}^n{\bar z}_jw^{\alpha}_j$,
and $\mbox{Res}_V$ is the residue of Coleff-Herrera \cite{CH}.
\end{proposition}

\section{Estimates for the solution operator.}
\label{IEstimates}

\indent
In this section we analyze the solution operators $I_q^{\epsilon}$, specifically estimates for limits
of those operators as $\epsilon\to 0$. In the estimates below we slightly abuse the notation by
using the same letter $C$ in all estimates for constants that do not depend on $\epsilon$, $\tau$
and $\eta$.\\
\indent
In the next lemma we simplify expression \eqref{IOperator} for
$I_q^{\epsilon}$ by eliminating the first integral in its right-hand side.

\begin{lemma}\label{BochnerZero} Let $V\subset \C\P^n$ be a reduced subvariety
as in \eqref{Variety}, and let $g_{\alpha}$ be an analytic function on
$U_{\alpha}\subset \C\P^n$ as in Theorem~\ref{HodgeRepresentation}.\\
\indent
Then for a fixed $\eta>0$ and an arbitrary $z\in U_{\alpha}$,
such that $|g_{\alpha}(z)|>\eta$, we have the following equality
\begin{equation}\label{ZeroBochner}
\lim_{t\to 0}\int_{U^{\epsilon(t)}\times[0,1]}
\vartheta_{\beta}(\zeta)\Phi(\zeta)
\wedge\omega^{\prime}_{q}\left((1-\lambda)\frac{\bar z}{B^*(\zeta,z)}
+\lambda\frac{\bar\zeta}{B(\zeta,z)}\right)\wedge\omega(\zeta)=0,
\end{equation}
where $\epsilon(t)$ is an admissible path, and $\beta\in (0,\dots,n)$.
\end{lemma}
\begin{proof} For a fixed $z$ with $|g_{\alpha}(z)|>\eta$ we choose $\tau>0$ so that for
$|\zeta-z|<\tau$ we have $|g_{\alpha}(\zeta)|>\eta/2$. Then we represent the integral
in \eqref{ZeroBochner} as
\begin{multline}\label{PartitionIntegral}
\int_{U^{\epsilon}\times[0,1]}
\vartheta_{\beta}(\zeta)\Phi(\zeta)
\wedge\omega^{\prime}_{q}\left((1-\lambda)\frac{\bar z}{B^*(\zeta,z)}
+\lambda\frac{\bar\zeta}{B(\zeta,z)}\right)\wedge\omega(\zeta)\\
=\int_{\left(U^{\epsilon}\cap\left\{|\zeta-z|<\tau\right\}\right)\times[0,1]}
\vartheta_{\beta}(\zeta)\Phi(\zeta)
\wedge\omega^{\prime}_{q}\left((1-\lambda)\frac{\bar z}{B^*(\zeta,z)}
+\lambda\frac{\bar\zeta}{B(\zeta,z)}\right)\wedge\omega(\zeta)\\
+\int_{\left(U^{\epsilon}\cap\left\{|\zeta-z|>\tau\right\}\right)\times[0,1]}
\vartheta_{\beta}(\zeta)\Phi(\zeta)
\wedge\omega^{\prime}_{q}\left((1-\lambda)\frac{\bar z}{B^*(\zeta,z)}
+\lambda\frac{\bar\zeta}{B(\zeta,z)}\right)\wedge\omega(\zeta).
\end{multline}
\indent
To estimate the first integral in the right-hand side of \eqref{PartitionIntegral}
we introduce the coordinates
\begin{equation}\label{SmallCoordinates}
\left\{\begin{aligned}
&t=\text{Im}B(\zeta,z)=\text{Im}B^*(\zeta,z),\\
&\rho_k=\left|P_k(\zeta)\right|\ \text{for}\ k=1\dots, m,
\end{aligned}\right.
\end{equation}
and obtain the following estimate
\begin{multline}\label{FirstBochner}
\left|\int_{\left(U^{\epsilon}\cap\left\{|\zeta-z|<\tau\right\}\right)\times[0,1]}
\vartheta_{\beta}(\zeta)\Phi(\zeta)
\wedge\omega^{\prime}_{q-1}\left((1-\lambda)\frac{\bar z}{B^*(\zeta,z)}
+\lambda\frac{\bar\zeta}{B(\zeta,z)}\right)\wedge\omega(\zeta)\right|\\
\leq C\cdot\int_0^{\tau}dt\int_0^{\epsilon}\rho_1 d\rho_1\cdots
\int_0^{\epsilon}\rho_m d\rho_m
\int_0^{\tau}\frac{r^{2(n-m)}dr}{\left(t+\sum_{k=1}^m\rho_k^2+r^2\right)^{n+1}}\\
\leq C\cdot\int_0^{\tau}dt\int_0^{\epsilon}\rho_1 d\rho_1\cdots \int_0^{\epsilon}\rho_m d\rho_m
\int_0^{\tau}\frac{dr}{\left(t+\sum_{k=1}^m\rho_k^2+r^2\right)^{m+1}}\\
\leq C\cdot\int_0^{\epsilon}\rho_1 d\rho_1\int_0^{\tau}dt
\int_0^{\tau}\frac{dr}{\left(t+\rho_1^2+r^2\right)^{2}}\\
\leq C\cdot\int_0^{\epsilon}d\rho_1\int_0^{\tau}\frac{\rho_1dr}{\rho_1^2+r^2}
\leq C\cdot\int_0^{\epsilon}d\rho_1\int_0^{\infty}\frac{du}{1+u^2}\leq C\cdot\epsilon\
\rightarrow\ 0
\end{multline}
as $\epsilon\to 0$.\\
\indent
For the second integral in the right-hand side of \eqref{PartitionIntegral} we have
\begin{equation}\label{SecondBochner}
\left|\int_{\left(U^{\epsilon}\cap\left\{|\zeta-z|>\tau\right\}\right)\times[0,1]}
\vartheta_{\beta}(\zeta)\Phi(\zeta)
\wedge\omega^{\prime}_{q}\left((1-\lambda)\frac{\bar z}{B^*(\zeta,z)}
+\lambda\frac{\bar\zeta}{B(\zeta,z)}\right)\wedge\omega(\zeta)\right|\rightarrow\ 0
\end{equation}
as $\epsilon\to 0$ because the integrand in \eqref{SecondBochner} is uniformly bounded for
$\left\{\zeta:\ |\zeta - z|>\tau\right\}$,
and the volume of $U^{\epsilon}$ goes to zero as $\epsilon\to 0$.\\
\indent
Combining estimates \eqref{FirstBochner} and \eqref{SecondBochner} we obtain the statement
of the Lemma.
\end{proof}

\indent
To estimate the rest of the integrals in \eqref{IOperator} we transform those integrals
by integrating the kernels with respect to variables $\lambda,\ \mu_j \in \Delta_J$ for $j\in J$
and obtain
\begin{multline}\label{IDeterminants}
\sum_{|J|\geq 1}\int_{\Gamma^{\epsilon}_J\times\Delta_J}
\Phi(\zeta)\wedge \omega^{\prime}_{q-1}\left((1-\lambda-\sum_{k=1}^m\mu_k)
\frac{\bar z}{B^*(\zeta,z)}
+\lambda\frac{\bar\zeta}{B(\zeta,z)}
+\sum_{j\in J}\mu_j\frac{Q_j(\zeta,z)}{P_j(\zeta)-P_j(z)}\right)
\wedge\omega(\zeta)\\
=\sum_{|J|\geq 1}C(n,q,|J|)\int_{\Gamma^{\epsilon}_J}\Phi(\zeta)
\bigwedge\det\left[\frac{\bar z}{B^*(\zeta,z)}\ \frac{\bar\zeta}{B(\zeta,z)}\
\overbrace{\frac{Q_j(\zeta,z)}{P_j(\zeta)-P_j(z)}}^{|J|}\
\overbrace{\frac{d{\bar z}}{B^*(\zeta,z)}}^{q-1}\
\overbrace{\frac{d{\bar\zeta}}{B(\zeta,z)}}^{n-|J|-q}\right]
\wedge\omega(\zeta).
\end{multline}

\indent
Then we further transform the integrals in the right-hand side of \eqref{IDeterminants}
using series representation \eqref{QSeries}
and obtain
\begin{multline}\label{SolutionIntegrals}
\int_{\Gamma^{\epsilon}_J}\Phi(\zeta)
\wedge\det\left[\frac{\bar z}{B^*(\zeta,z)}\ \frac{\bar\zeta}{B(\zeta,z)}\
\overbrace{\frac{Q_j(\zeta,z)}{P_j(\zeta)-P_j(z)}}^{|J|}\
\overbrace{\frac{d{\bar z}}{B^*(\zeta,z)}}^{q-1}\
\overbrace{\frac{d{\bar\zeta}}{B(\zeta,z)}}^{n-|J|-q}\right]
\wedge\omega(\zeta)\\
=\sum_{|A|\geq 0}
\int_{\Gamma^{\epsilon}_J}
\frac{\Phi(\zeta)}{\prod_{j\in J}P_j(\zeta)}
\cdot\frac{P^A(z)}{P^A(\zeta)}
\wedge\det\left[\frac{\bar z}{B^*(\zeta,z)}\ \frac{\bar\zeta}{B(\zeta,z)}\
\overbrace{Q_j(\zeta,z)}^{|J|}\
\overbrace{\frac{d{\bar z}}{B^*(\zeta,z)}}^{q-1}\
\overbrace{\frac{d{\bar\zeta}}{B(\zeta,z)}}^{n-|J|-q}\right]\wedge\omega(\zeta),
\end{multline}
where we assume that $J=(j_1,\dots,j_p)$, denote by $A=(a_1,\dots,a_p)$ a multiindex, by
$$P^A(\zeta)=P_{j_1}^{a_1}(\zeta)\cdots P_{j_p}^{a_p}(\zeta),$$
and by $|A|=a_1+\cdots+a_p$.\\
\indent
As before in \eqref{SeriesRepresentation}, using Theorem 1.7.6(2) from \cite{CH} (see also \cite{HP3} Prop. 2.3)
we obtain that the residual currents defined by the terms
in the right-hand side of \eqref{SolutionIntegrals} with $|A|\geq 1$
are zero-currents from the point of view of \eqref{preL-Operator}, and therefore
we can simplify formula \eqref{IOperator} for $I_q^{\epsilon(t)}\left[\Phi\right]$ as follows
\begin{multline}\label{I-Sum-of-Determinants}
I_q^{\epsilon(t)}\left[\Phi\right](z)
=\sum_{|J|\geq 1}C(n,q,|J|)\int_{\Gamma^{\epsilon(t)}_J}
\frac{\Phi(\zeta)}{\prod_{j\in J}P_j(\zeta)}\\
\bigwedge\det\left[\frac{\bar z}{B^*(\zeta,z)}\ \frac{\bar\zeta}{B(\zeta,z)}\
\overbrace{Q_j(\zeta,z)}^{|J|}\
\overbrace{\frac{d{\bar z}}{B^*(\zeta,z)}}^{q-1}\
\overbrace{\frac{d{\bar\zeta}}{B(\zeta,z)}}^{n-|J|-q}\right]\wedge\omega(\zeta).
\end{multline}

\indent
In the next lemma we further simplify formula \eqref{I-Sum-of-Determinants} for $I_q^{\epsilon(t)}$.

\begin{lemma}\label{ZeroTerms} Let $V\subset \C\P^n$ be a reduced subvariety
as in \eqref{Variety}, and let $g_{\alpha}$ be an analytic function on
$U_{\alpha}\subset \C\P^n$ as in Theorem~\ref{HodgeRepresentation}.\\
\indent
Then for a fixed $\eta>0$, an arbitrary $z\in U_{\alpha}$,
such that $|g_{\alpha}(z)|>\eta$, and $J$, such that $|J|=p<m$ we have
\begin{equation}\label{ZeroDeterminants}
\lim_{t\to 0}\int_{\Gamma^{\epsilon(t)}_J}
\frac{\vartheta_{\beta}(\zeta)\Phi(\zeta)}{\prod_{j\in J}P_j(\zeta)}
\bigwedge\det\left[\frac{\bar z}{B^*(\zeta,z)}\ \frac{\bar\zeta}{B(\zeta,z)}\
\overbrace{Q_j(\zeta,z)}^{|J|}\
\overbrace{\frac{d{\bar z}}{B^*(\zeta,z)}}^{q-1}\
\overbrace{\frac{d{\bar\zeta}}{B(\zeta,z)}}^{n-|J|-q}\right]\wedge\omega(\zeta)=0,
\end{equation}
for an admissible path $\epsilon(t)$, and $\beta\in (0,\dots,n)$.
\end{lemma}
\begin{proof}  For a fixed $z$ with $|g_{\alpha}(z)|>\eta$ we choose $\tau>0$ as in
Lemma~\ref{BochnerZero}, so that for
$|\zeta-z|<\tau$ we have $|g_{\alpha}(\zeta)|>\eta/2$.  Then we represent the integral
in \eqref{ZeroDeterminants} as
\begin{multline}\label{ZeroBochnerPartitioned}
\int_{\Gamma^{\epsilon(t)}_J}
\frac{\vartheta_{\beta}(\zeta)\Phi(\zeta)}{\prod_{j\in J}P_j(\zeta)}
\bigwedge\det\left[\frac{\bar z}{B^*(\zeta,z)}\ \frac{\bar\zeta}{B(\zeta,z)}\
\overbrace{Q_j(\zeta,z)}^{|J|}\
\overbrace{\frac{d{\bar z}}{B^*(\zeta,z)}}^{q-1}\
\overbrace{\frac{d{\bar\zeta}}{B(\zeta,z)}}^{n-|J|-q}\right]\wedge\omega(\zeta)\\
=\int_{\Gamma^{\epsilon(t)}_J\cap\left\{|\zeta-z|<\tau\right\}}
\frac{\vartheta_{\beta}(\zeta)\Phi(\zeta)}{\prod_{j\in J}P_j(\zeta)}
\bigwedge\det\left[\frac{\bar z}{B^*(\zeta,z)}\ \frac{\bar\zeta}{B(\zeta,z)}\
\overbrace{Q_j(\zeta,z)}^{|J|}\
\overbrace{\frac{d{\bar z}}{B^*(\zeta,z)}}^{q-1}\
\overbrace{\frac{d{\bar\zeta}}{B(\zeta,z)}}^{n-|J|-q}\right]\wedge\omega(\zeta)\\
+\int_{\Gamma^{\epsilon(t)}_J\cap\left\{|\zeta-z|>\tau\right\}}
\frac{\vartheta_{\beta}(\zeta)\Phi(\zeta)}{\prod_{j\in J}P_j(\zeta)}
\bigwedge\det\left[\frac{\bar z}{B^*(\zeta,z)}\ \frac{\bar\zeta}{B(\zeta,z)}\
\overbrace{Q_j(\zeta,z)}^{|J|}\
\overbrace{\frac{d{\bar z}}{B^*(\zeta,z)}}^{q-1}\
\overbrace{\frac{d{\bar\zeta}}{B(\zeta,z)}}^{n-|J|-q}\right]\wedge\omega(\zeta).
\end{multline}
For the first integral in the right-hand side of \eqref{ZeroBochnerPartitioned} 
using the coordinates from \eqref{SmallCoordinates} and estimate
$$\left|{\bar z}\wedge \bar{\zeta}\right|\leq C\cdot |\zeta-z|$$
we obtain
\begin{multline}\label{FirstZeroBochner}
\left|\int_{\Gamma^{\epsilon(t)}_J\cap\left\{|\zeta-z|<\tau\right\}}
\frac{\vartheta_{\beta}(\zeta)\Phi(\zeta)}{\prod_{j\in J}P_j(\zeta)}
\bigwedge\det\left[\frac{\bar z}{B^*(\zeta,z)}\ \frac{\bar\zeta}{B(\zeta,z)}\
\overbrace{Q_j(\zeta,z)}^{|J|}\
\overbrace{\frac{d{\bar z}}{B^*(\zeta,z)}}^{q-1}\
\overbrace{\frac{d{\bar\zeta}}{B(\zeta,z)}}^{n-|J|-q}\right]\wedge\omega(\zeta)\right|\\
\leq C\cdot\int_0^{\tau}dt\int_0^{2\pi} d\phi_{j_1}\cdots\int_0^{2\pi} d\phi_{j_p}
\int_0^{\epsilon}\rho_1 d\rho_1\cdots\int_0^{\epsilon}\rho_{m-p}d\rho_{m-p}\\
\times\int_0^{\tau}\frac{r^{2(n-m)}dr}
{\left(t+\sum_{k=1}^{m-p}\rho_k^2+r^2\right)^{n-p+1}}\\
\leq C\cdot\int_0^{\tau}dt
\int_0^{\epsilon}\rho_1 d\rho_1\cdots\int_0^{\epsilon}\rho_{m-p}d\rho_{m-p}
\int_0^{\tau}\frac{dr}{\left(t+\sum_{k=1}^{m-p}\rho_k^2+r^2\right)^{m-p+1}}\\
\leq C\cdot\int_0^{\tau}dt\int_0^{\epsilon}\rho_1 d\rho_1
\int_0^{\tau}\frac{dr}{\left(t+\rho_1^2+r^2\right)^2}\\
\leq C\cdot\int_0^{\epsilon}\rho_1 d\rho_1
\int_0^{\tau}\frac{dr}{\left(\rho_1^2+r^2\right)}
\leq C\cdot\int_0^{\epsilon}\rho_1 d\rho_1
\int_0^{\tau}\frac{dr}{\rho_1^2\left(1+(r/\rho_1)^2\right)}\\
\leq C\cdot\int_0^{\epsilon}d\rho_1\int_0^{\infty}\frac{du}{1+u^2}
\leq C\cdot\epsilon \rightarrow 0,
\end{multline}
as $\epsilon\to 0$.\\
\indent
For the second integral in \eqref{ZeroBochnerPartitioned} we may assume without loss of generality
that $\beta=0$ and $\Phi$ is a smooth differential form with support in $\left\{|\zeta-z|>\tau\right\}$,
and rewrite this integral using polynomials $\left\{F_k\right\}_1^m$ from \eqref{Polynomials} as
\begin{multline*}
\int_{\Gamma^{\epsilon(t)}_J\cap\left\{|\zeta-z|>\tau\right\}}
\frac{\vartheta_{\beta}(\zeta)\Phi(\zeta)}{\prod_{j\in J}P_j(\zeta)}
\bigwedge\det\left[\frac{\bar z}{B^*(\zeta,z)}\ \frac{\bar\zeta}{B(\zeta,z)}\
\overbrace{Q_j(\zeta,z)}^{|J|}\
\overbrace{\frac{d{\bar z}}{B^*(\zeta,z)}}^{q-1}\
\overbrace{\frac{d{\bar\zeta}}{B(\zeta,z)}}^{n-|J|-q}\right]\wedge\omega(\zeta)\\
=\lim_{t\to 0}\int_{\footnotesize\left\{\begin{array}{ll}
|F_j(\zeta)|\cdot\chi_j(\zeta)=\epsilon_j(t)\ \text{for}\ j\in J,\\
|F_k(\zeta)|\cdot\chi_k(\zeta)<\epsilon_k(t)\ \text{for}\ k\notin J
\end{array}\right\}}
\frac{\Psi(\zeta,z)}{\prod_{j\in J}F_j(\zeta)}
\end{multline*}
with a smooth form $\Psi(\zeta,z)$ analytically depending on $(z, {\bar z})$ for
$z\in\left\{U_{\alpha}:\ |g_{\alpha}(z)|>\eta\right\}$ and compact support
in $\left\{|\zeta-z|>\tau\right\}$.\\
\indent
Then as in Lemma~\ref{LimitsExistence} we obtain equality
\begin{multline*}
\lim_{t\to 0}\int_{\footnotesize\left\{\begin{array}{ll}
|F_j(\zeta)|\cdot\chi_j(\zeta)=\epsilon_j(t)\ \text{for}\ j\in J,\\
|F_k(\zeta)|\cdot\chi_k(\zeta)<\epsilon_k(t)\ \text{for}\ k\notin J
\end{array}\right\}}
\frac{\Psi(\zeta,z)}{\prod_{j\in J}F_j(\zeta)}\\
=\lim_{\eta\to 0}\lim_{t\to 0}
\int_{\footnotesize\left\{|g_{\beta}(\zeta)|>\eta,\ \begin{array}{ll}
|F_j(\zeta)|\cdot\chi_j(\zeta)=\epsilon_j(t)\ \text{for}\ j\in J,\\
|F_k(\zeta)|\cdot\chi_k(\zeta)<\epsilon_k(t)\ \text{for}\ k\notin J
\end{array}\right\}}
\frac{\Psi(\zeta,z)}{\prod_{j\in J}F_j(\zeta)}
\end{multline*}
reducing the proof of \eqref{ZeroDeterminants} to the proof of equality
\begin{equation}\label{FiberZero}
\lim_{t\to 0}\int_{\footnotesize\left\{|g_{\beta}(\zeta)|>\eta,\ \begin{array}{ll}
|F_j(\zeta)|\cdot\chi_j(\zeta)=\epsilon_j(t)\ \text{for}\ j\in J,\\
|F_k(\zeta)|\cdot\chi_k(\zeta)<\epsilon_k(t)\ \text{for}\ k\notin J
\end{array}\right\}}
\frac{\Psi(\zeta,z)}{\prod_{j\in J}F_j(\zeta)}=0.
\end{equation}
\indent
In the proof of \eqref{FiberZero} we use the method used in Lemma 2.3 in \cite{HP4}. Namely,
localizing the problem we assume that the set $V\cap\left\{|g_{\beta}(z)|>\eta\right\}$
is a submanifold in a polydisk ${\cal P}^n$ of the form
$$S=\left\{u\in{\cal P}^n:\ u_1=\cdots=u_m=0\right\},$$
and the integral in \eqref{FiberZero} can be represented as
\begin{equation*}
\lim_{t\to 0}\int_{\footnotesize\left\{\begin{array}{ll}
|u_j|\cdot\chi_j(u)=\epsilon_j(t)\ \text{for}\ j\in J,\\
|u_k|\cdot\chi_k(u)<\epsilon_k(t)\ \text{for}\ k\notin J
\end{array}\right\}}
\frac{f(u,z)}{\prod_{j\in J}u_j}=0.
\end{equation*}
\end{proof}

\indent
In the next lemma we obtain an explicit form of a solution operator for $\bar\partial$-equation
on residual currents.

\begin{lemma}\label{BochnerIntegrals} Let $V\subset \C\P^n$ be a reduced
subvariety as in \eqref{Variety}, and let $g_{\alpha}$ be an analytic function on $U_{\alpha}\subset \C\P^n$ as in
Theorem~\ref{HodgeRepresentation}.\\
\indent
Then for a fixed $\eta>0$, $J=(1,\dots,m)$, an admissible path $\epsilon(t)$,
and $\beta\in(0,\dots,n)$ we have:
\begin{itemize}
\item[(i)] the limits of integrals
\begin{equation}\label{BochnerDefined}
\lim_{t\to 0}\int_{\Gamma^{\epsilon(t)}_J}
\frac{\vartheta_{\beta}(\zeta)\Phi(\zeta)}{\prod_{k=1}^mP_k(\zeta)}
\bigwedge\det\left[\frac{\bar z}{B^*(\zeta,z)}\ \frac{\bar\zeta}{B(\zeta,z)}\
\overbrace{Q(\zeta,z)}^{m}\
\overbrace{\frac{d{\bar z}}{B^*(\zeta,z)}}^{q-1}\
\overbrace{\frac{d{\bar\zeta}}{B(\zeta,z)}}^{n-m-q}\right]\wedge\omega(\zeta)
\end{equation}
are well defined continuous functions on $\left\{U_{\alpha}:\ |g_{\alpha}(z)|>\eta\right\}$,
\item[(ii)]
if
\begin{equation}\label{PhiClosed}
\Phi(\zeta)\Big|_{U_{\beta}}=\sum_{k=1}^m F^{(\beta)}_k(\zeta)\cdot \Psi_k(\zeta),
\end{equation}
where $\left\{F^{(\beta)}_k\right\}_1^m$ are the polynomials from \eqref{Polynomials},
then the limit in \eqref{BochnerDefined} is equal to zero.
\end{itemize}
\end{lemma}
\begin{proof} For a fixed $z$ with $|g_{\alpha}(z)|>\eta$ we choose $\tau>0$ so that for
$|\zeta-z|<\tau$ we have $|g_{\alpha}(\zeta)|>\eta/2$. Then for $\epsilon<\tau$
we represent the integral in \eqref{BochnerDefined} as
\begin{multline}\label{BochnerPartitioned}
\int_{\Gamma^{\epsilon}_J}
\frac{\vartheta_{\beta}(\zeta)\Phi(\zeta)}{\prod_{k=1}P_k(\zeta)}
\bigwedge\det\left[\frac{\bar z}{B^*(\zeta,z)}\ \frac{\bar\zeta}{B(\zeta,z)}\
\overbrace{Q_j(\zeta,z)}^{m}\
\overbrace{\frac{d{\bar z}}{B^*(\zeta,z)}}^{q-1}\
\overbrace{\frac{d{\bar\zeta}}{B(\zeta,z)}}^{n-m-q}\right]\wedge\omega(\zeta)\\
=\int_{\Gamma^{\epsilon}_J\cap\left\{|\zeta-z|<\sqrt{\epsilon}\right\}}
\frac{\vartheta_{\beta}(\zeta)\Phi(\zeta)}{\prod_{k=1}^mP_k(\zeta)}
\bigwedge\det\left[\frac{\bar z}{B^*(\zeta,z)}\ \frac{\bar\zeta}{B(\zeta,z)}\
\overbrace{Q_j(\zeta,z)}^{m}\
\overbrace{\frac{d{\bar z}}{B^*(\zeta,z)}}^{q-1}\
\overbrace{\frac{d{\bar\zeta}}{B(\zeta,z)}}^{n-m-q}\right]\wedge\omega(\zeta)\\
+\int_{\Gamma^{\epsilon}_J\cap\left\{|\zeta-z|>\sqrt{\epsilon}\right\}}
\frac{\vartheta_{\beta}(\zeta)\Phi(\zeta)}{\prod_{k=1}^mP_k(\zeta)}
\bigwedge\det\left[\frac{\bar z}{B^*(\zeta,z)}\ \frac{\bar\zeta}{B(\zeta,z)}\
\overbrace{Q_j(\zeta,z)}^{m}\
\overbrace{\frac{d{\bar z}}{B^*(\zeta,z)}}^{q-1}\
\overbrace{\frac{d{\bar\zeta}}{B(\zeta,z)}}^{n-m-q}\right]\wedge\omega(\zeta).
\end{multline}
\indent
For the first integral in the right-hand side of \eqref{BochnerPartitioned} we have
\begin{multline}\label{FirstBochnerBounded}
\left|\int_{\Gamma^{\epsilon}_J\cap\left\{|\zeta-z|<\sqrt{\epsilon}\right\}}
\frac{\vartheta_{\beta}(\zeta)\Phi(\zeta)}{\prod_{k=1}^mP_k(\zeta)}
\bigwedge\det\left[\frac{\bar z}{B^*(\zeta,z)}\ \frac{\bar\zeta}{B(\zeta,z)}\
\overbrace{Q(\zeta,z)}^{m}\
\overbrace{\frac{d{\bar z}}{B^*(\zeta,z)}}^{q-1}\
\overbrace{\frac{d{\bar\zeta}}{B(\zeta,z)}}^{n-m-q}\right]\wedge\omega(\zeta)\right|\\
\leq C\cdot\int_0^{\sqrt{\epsilon}}dt\int_0^{2\pi} d\phi_1\cdots\int_0^{2\pi} d\phi_m
\int_0^{\sqrt{\epsilon}}\frac{r^{2(n-m)}dr}
{\left(t+r^2\right)^{n-m+1}}\\
\leq C\cdot\int_0^{\sqrt{\epsilon}}dt\int_0^{\sqrt{\epsilon}}\frac{dr}{t+r^2}
\leq C\cdot\int_0^{\sqrt{\epsilon}}\frac{dt}{\sqrt{t}}\int_0^{\infty}\frac{du}{1+u^2}
\leq C\cdot\sqrt[4]{\epsilon}\rightarrow 0,\\
\end{multline}
as $\epsilon\to 0$.\\
\indent
For the second integral in \eqref{BochnerPartitioned} we denote
$$K(\zeta,z)=\frac{\vartheta_{\beta}(\zeta)\Phi(\zeta)}{\prod_{k=1}^mP_k(\zeta)}
\bigwedge\det\left[\frac{\bar z}{B^*(\zeta,z)}\ \frac{\bar\zeta}{B(\zeta,z)}\
\overbrace{Q(\zeta,z)}^{m}\
\overbrace{\frac{d{\bar z}}{B^*(\zeta,z)}}^{q-1}\
\overbrace{\frac{d{\bar\zeta}}{B(\zeta,z)}}^{n-m-q}\right]\wedge\omega(\zeta)$$
and consider $z^{(1)}, z^{(2)}$ such that $|z^{(1)}-z^{(2)}|<\epsilon/4$.\\
\indent
Then using relations
$$\begin{array}{ll}
|\zeta-z|>\sqrt{\epsilon}\Longrightarrow \left|B(\zeta,z)\right|>\epsilon/2,
\vspace{0.1in}\\
|z^{(1)}-z^{(2)}|<\epsilon/4,\hspace{0.1in}\left|B(\zeta,z^{(1)})\right|>\epsilon/2
\Longrightarrow \left|B(\zeta,z^{(2)})\right|>\epsilon/4,
\end{array}$$
we obtain
\begin{multline}\label{BigB}
\left|\int_{\Gamma^{\epsilon}_J\cap\left\{|\zeta-z^{(1)}|>\sqrt{\epsilon}\right\}}
\left(K(\zeta,z^{(1)})-K(\zeta,z^{(2)})\right)\right|\\
\leq\left|\int_{\Gamma^{\epsilon}_J\cap\left\{\left|B(\zeta,z^{(1)})\right|>\epsilon/2\right\}}
\left(K(\zeta,z^{(1)})-K(\zeta,z^{(2)})\right)\right|\\
\leq C\epsilon\cdot\int_0^{A}dt
\int_0^{2\pi} d\phi_1\cdots\int_0^{2\pi} d\phi_m
\int_0^{A}\frac{r^{2(n-m)}dr}{\left(\epsilon+t+r^2\right)^{n-m+2}}\\
\leq C\epsilon\cdot\int_0^{A}dt\int_0^{A}
\frac{dr}{\left(\epsilon+t+r^2\right)^2}
\leq C\epsilon\cdot \int_0^{A}\frac{dr}{\left(\epsilon+r^2\right)}\\
\leq C\epsilon\cdot \int_0^{A}
\frac{dr}{\epsilon\left(1+\left(r/\sqrt{\epsilon}\right)^2\right)}
\leq C\sqrt{\epsilon}\cdot \int_0^{\infty}\frac{du}{1+u^2}\leq C\sqrt{\epsilon}.
\end{multline}
\indent
From estimates \eqref{FirstBochnerBounded} and \eqref{BigB} we obtain claim (i)
of the Lemma.\\
\indent
Claim (ii) of the Lemma for the first integral in \eqref{BochnerPartitioned} follows
from estimate \eqref{FirstBochnerBounded}, and for the second integral in
\eqref{BochnerPartitioned} it follows from Lemma~\ref{LimitsExistence} with additional
application of Theorem 1.7.6(2) from \cite{CH} (see also \cite{HP3} Prop. 2.3).
\end{proof}
\indent
In the proposition below we prove smoothness of limits of integrals in \eqref{BochnerDefined}
under assumption of smoothness of the current $\phi$.

\begin{proposition}\label{Smoothness} Let $V\subset \C\P^n$ be a reduced subvariety as in \eqref{Variety}, let $g_{\alpha}$ be an analytic function on $U_{\alpha}\subset \C\P^n$
as in Theorem~\ref{HodgeRepresentation}, and let the $\bar\partial$-closed current
$\phi$ be defined by a $C^{\infty}$ form $\Phi$.\\
\indent
Then for a fixed $\eta>0$, $J=(1,\dots,m)$, an admissible path $\epsilon(t)$,
the limits of integrals in \eqref{BochnerDefined} represent $C^{\infty}$ forms
on $\left\{U_{\alpha}:\ |g_{\alpha}(z)|>\eta \right\}$.
\end{proposition}
\begin{proof} For a fixed $z$ with $|g_{\alpha}(z)|>\eta$ we choose $\tau>0$ as in
Lemma~\ref{BochnerZero}, so that for
$|\zeta-z|<\tau$ we have $|g_{\alpha}(\zeta)|>\eta/2$. Then, as in Lemma~\ref{ZeroTerms}, we represent the integral in \eqref{ZeroDeterminants} as
\begin{multline}\label{SmoothnessPartitioned}
\int_{\Gamma^{\epsilon(t)}_J}
\frac{\vartheta_{\beta}(\zeta)\Phi(\zeta)}{\prod_{k=1}^mP_k(\zeta)}
\bigwedge\det\left[\frac{\bar z}{B^*(\zeta,z)}\ \frac{\bar\zeta}{B(\zeta,z)}\
\overbrace{Q(\zeta,z)}^{m}\
\overbrace{\frac{d{\bar z}}{B^*(\zeta,z)}}^{q-1}\
\overbrace{\frac{d{\bar\zeta}}{B(\zeta,z)}}^{n-m-q}\right]\wedge\omega(\zeta)\\
=\int_{\Gamma^{\epsilon(t)}_J\cap\left\{|\zeta-z|<\tau\right\}}
\frac{\vartheta_{\beta}(\zeta)\Phi(\zeta)}{\prod_{k=1}^mP_k(\zeta)}
\bigwedge\det\left[\frac{\bar z}{B^*(\zeta,z)}\ \frac{\bar\zeta}{B(\zeta,z)}\
\overbrace{Q(\zeta,z)}^{m}\
\overbrace{\frac{d{\bar z}}{B^*(\zeta,z)}}^{q-1}\
\overbrace{\frac{d{\bar\zeta}}{B(\zeta,z)}}^{n-m-q}\right]\wedge\omega(\zeta)\\
+\int_{\Gamma^{\epsilon(t)}_J\cap\left\{|\zeta-z|>\tau\right\}}
\frac{\vartheta_{\beta}(\zeta)\Phi(\zeta)}{\prod_{k=1}^mP_k(\zeta)}
\bigwedge\det\left[\frac{\bar z}{B^*(\zeta,z)}\ \frac{\bar\zeta}{B(\zeta,z)}\
\overbrace{Q(\zeta,z)}^{m}\
\overbrace{\frac{d{\bar z}}{B^*(\zeta,z)}}^{q-1}\
\overbrace{\frac{d{\bar\zeta}}{B(\zeta,z)}}^{n-m-q}\right]\wedge\omega(\zeta).
\end{multline}
\indent
From Theorem 1.7.2 in \cite{CH} (see also \cite{HP4} Prop. 2.2) we obtain that the second integral
in the right-hand side of \eqref{SmoothnessPartitioned} represents a $C^{\infty}$ form with respect to $z$,
since the functions $|B(\zeta,z)|, |B^*(\zeta,z)|$ are separated from zero uniformly with
respect to $z$ for $|\zeta-z|>\tau$.

\indent
To prove the statement of the proposition for the first integral in the right-hand side
of \eqref{SmoothnessPartitioned} we use the following lemma.

\begin{lemma}\label{LocalSmoothness} Let $V\subset \C\P^n$ be a reduced subvariety
as in \eqref{Variety}, let $z\in V$ be a nonsingular point, and let $\Phi$
be a $C^{l+1}$ form with compact support in a neighborhood $U_z\ni z$, such that $U_z\cap\text{sing}V=\emptyset$.\\
\indent
Then for $J=(1,\dots,m)$ and an admissible path $\epsilon(t)$ the limit
\begin{equation}\label{SmoothIntegral}
\lim_{t\to 0}\int_{\Gamma^{\epsilon(t)}_J\cap\ U}
\frac{\Phi(\zeta)\cdot\overbrace{Q(\zeta,z)}^{m}}{\prod_{k=1}^mP_k(\zeta)}
\wedge\frac{\overbrace{d{\bar z}}^{q-1}\wedge\overbrace{d{\bar\zeta}}^{n-m-q}}
{\left(B^*(\zeta,z)\right)^{q}\left(B(\zeta,z)\right)^{n-m-q+1}}\wedge\omega(\zeta),
\end{equation}
defines a $C^{l}$-form on some neighborhood of $V$.
\end{lemma}
\begin{proof} Using equalities
\begin{equation}\label{BarZEqualities}
\begin{aligned}
&\frac{\partial}{\partial{\bar z}_k}
\left[\frac{1}{\left(B^*(\zeta,z)\right)^p}\right]
=\frac{\partial}{\partial{\bar z}_k}
\left[\frac{1}{\left(-1+\sum_{j=0}^n{\bar z}_j\zeta_j\right)^{p}}\right]
=-p\left[\frac{\zeta_k}{\left(B^*(\zeta,z)\right)^{p+1}}\right],\\
&\frac{\partial}{\partial{\bar z}_k}
\left[\frac{1}{\left(B(\zeta,z)\right)^p}\right]
=\frac{\partial}{\partial{\bar z}_k}
\left[\frac{1}{\left(1-\sum_{j=0}^n{\bar\zeta}_jz_j\right)^{p}}\right]=0,
\end{aligned}
\end{equation}
we obtain
\begin{multline*}
\frac{\partial^l}{\partial{\bar z}_k^l}\int_{\Gamma^{\epsilon(t)}_J\cap\ U}
\frac{\Phi(\zeta)\cdot\overbrace{Q(\zeta,z)}^{m}}{\prod_{k=1}^mP_k(\zeta)}
\wedge\frac{\overbrace{d{\bar z}}^{q-1}\wedge\overbrace{d{\bar\zeta}}^{n-m-q}}
{\left(B^*(\zeta,z)\right)^{q}\left(B(\zeta,z)\right)^{n-m-q+1}}\wedge\omega(\zeta)\\
=\int_{\Gamma^{\epsilon(t)}_J\cap\ U}
\frac{\Phi(\zeta)\cdot\overbrace{Q(\zeta,z)}^{m}}{\prod_{k=1}^mP_k(\zeta)}
\cdot\frac{\partial^l}{\partial{\bar z}_k^l}\left[\frac{1}{\left(B^*(\zeta,z)\right)^{q}}\right]
\wedge\frac{\overbrace{d{\bar z}}^{q-1}\wedge\overbrace{d{\bar\zeta}}^{n-m-q}}
{\left(B(\zeta,z)\right)^{n-m-q+1}}\wedge\omega(\zeta)\\
=(-1)^lq\cdots(q+l-1)\int_{\Gamma^{\epsilon(t)}_J\cap\ U}
\frac{\Phi(\zeta)\cdot\overbrace{Q(\zeta,z)}^{m}}{\prod_{k=1}^mP_k(\zeta)}
\wedge\frac{\zeta_k^l\cdot\overbrace{d{\bar z}}^{q-1}
\wedge\overbrace{d{\bar\zeta}}^{n-m-q}}
{\left(B^*(\zeta,z)\right)^{q+l}\left(B(\zeta,z)\right)^{n-m-q+1}}\wedge\omega(\zeta).
\end{multline*}
\indent
To estimate the right-hand side of the equality above we assume without loss of
generality that for some $i\in (0,\dots,n)$ we have $z_i\neq 0$ in $U_z$. Then using equalities
\begin{equation}\label{ZetaEqualities}
\begin{aligned}
&\frac{\partial}{\partial\zeta_k}
\left[\frac{1}{\left(B^*(\zeta,z)\right)^{p}}\right]
=\frac{\partial}{\partial\zeta_k}
\left[\frac{1}{\left(-1+\sum_{j=0}^n{\bar z}_j\zeta_j\right)^{p}}\right]
=-p\left[\frac{{\bar z}_k}{\left(B^*(\zeta,z)\right)^{p+1}}\right],\\
&\frac{\partial}{\partial\zeta_k}
\left[\frac{1}{\left(B(\zeta,z)\right)^{p}}\right]
=\frac{\partial}{\partial\zeta_k}
\left[\frac{1}{\left(1-\sum_{j=0}^n{\bar\zeta}_jz_j\right)^{p}}\right]=0,
\end{aligned}
\end{equation}
we obtain for $q>1$
\begin{multline*}
(-1)^lq\cdots(q+l-1){\bar z}_i^{l+1}\int_{\Gamma^{\epsilon(t)}_J\cap\ U}
\frac{\Phi(\zeta)\cdot\overbrace{Q(\zeta,z)}^{m}}{\prod_{k=1}^mP_k(\zeta)}
\wedge\frac{\zeta_k^l\cdot\overbrace{d{\bar z}}^{q-1}
\wedge\overbrace{d{\bar\zeta}}^{n-m-q}}
{\left(B^*(\zeta,z)\right)^{q+l}\left(B(\zeta,z)\right)^{n-m-q+1}}\wedge\omega(\zeta)\\
=-\frac{1}{q-1}\int_{\Gamma^{\epsilon(t)}_J\cap\ U}
\frac{\Phi(\zeta)\cdot\overbrace{Q(\zeta,z)}^{m}}{\prod_{k=1}^mP_k(\zeta)}
\frac{\partial^{l+1}}{\partial\zeta_i^{l+1}}
\left[\frac{1}{\left(B^*(\zeta,z)\right)^{q-1}}\right]
\wedge\frac{\zeta_k^l\cdot\overbrace{d{\bar z}}^{q-1}
\wedge\overbrace{d{\bar\zeta}}^{n-m-q}}
{\left(B(\zeta,z)\right)^{n-m-q+1}}\wedge\omega(\zeta)\\
=-\frac{1}{q-1}\int_{\Gamma^{\epsilon(t)}_J\cap\ U}
\frac{\partial^{l+1}}{\partial\zeta_i^{l+1}}
\left[\frac{\Phi(\zeta)\cdot\overbrace{Q(\zeta,z)}^{m}}{\prod_{k=1}^mP_k(\zeta)}\right]
\wedge\frac{\zeta_k^l\cdot\overbrace{d{\bar z}}^{q-1}
\wedge\overbrace{d{\bar\zeta}}^{n-m-q}}
{\left(B^*(\zeta,z)\right)^{q-1}\left(B(\zeta,z)\right)^{n-m-q+1}}
\wedge\omega(\zeta),
\end{multline*}
where in the last equality we used integration by parts.\\
\indent
Estimate similar to \eqref{FirstBochnerBounded} produces the following estimate of the integral
in the right-hand side of equality above
\begin{multline*}
\left|\int_{\Gamma^{\epsilon(t)}_J\cap\ U}
\frac{\partial^{l+1}}{\partial\zeta_i^{l+1}}
\left[\frac{\Phi(\zeta)\cdot\overbrace{Q(\zeta,z)}^{m}}{\prod_{k=1}^mP_k(\zeta)}\right]
\wedge\frac{\zeta_k^l\cdot\overbrace{d{\bar z}}^{q-1}
\wedge\overbrace{d{\bar\zeta}}^{n-m-q}}
{\left(B^*(\zeta,z)\right)^{q-1}\left(B(\zeta,z)\right)^{n-m-q+1}}
\wedge\omega(\zeta)\right|\\
\leq C\cdot\left\|\Phi\right\|_{C^{l+1}}
\int_0^{\tau}dt\int_0^{2\pi} d\phi_1\cdots\int_0^{2\pi} d\phi_m
\int_0^{\tau}\frac{r^{2(n-m)-1}}{\left(t+r^2\right)^{n-m}}dr\\
\leq C\cdot\left\|\Phi\right\|_{C^{l+1}}\int_0^{\tau}dt\int_0^{\tau}\frac{rdr}{t+r^2}
\leq C\cdot\left\|\Phi\right\|_{C^{l+1}}.\\
\end{multline*}
\indent
We notice that the same estimate is valid for $q=1$ if we use the fact that the functions
$$\log{\left(-1+\sum_{j=0}^n{\bar z}_j\zeta_j\right)},\hspace{0.2in}
\log{\left(1-\sum_{j=0}^n{\bar\zeta}_jz_j\right)}$$
are well defined on $\*S(1)\times\*S(1)$, and satisfy equations similar to the ones used above.\\
\indent
We obtain similar estimates for mixed derivatives
${\dis \frac{\partial^l}{\partial z^s\partial{\bar z}^p} }$ using together with equalities
\eqref{BarZEqualities} and \eqref{ZetaEqualities} the equalities
\begin{equation*}
\begin{aligned}
&\frac{\partial}{\partial z_k}
\left[\frac{1}{\left(B^*(\zeta,z)\right)^p}\right]
=\frac{\partial}{\partial z_k}
\left[\frac{1}{\left(-1+\sum_{j=0}^n{\bar z}_j\zeta_j\right)^{p}}\right]=0,\\
&\frac{\partial}{\partial z_k}
\left[\frac{1}{\left(B(\zeta,z)\right)^p}\right]
=\frac{\partial}{\partial z_k}
\left[\frac{1}{\left(1-\sum_{j=0}^n{\bar\zeta}_jz_j\right)^{p}}\right]
=p\left[\frac{{\bar\zeta}_k}{\left(B(\zeta,z)\right)^{p+1}}\right],\\
\end{aligned}
\end{equation*}
and
\begin{equation*}
\begin{aligned}
&\frac{\partial}{\partial{\bar\zeta}_k}
\left[\frac{1}{\left(B^*(\zeta,z)\right)^{p}}\right]
=\frac{\partial}{\partial{\bar\zeta}_k}
\left[\frac{1}{\left(-1+\sum_{j=0}^n{\bar z}_j\zeta_j\right)^{p}}\right]=0,\\
&\frac{\partial}{\partial{\bar\zeta}_k}
\left[\frac{1}{\left(B(\zeta,z)\right)^{p}}\right]
=\frac{\partial}{\partial{\bar\zeta}_k}
\left[\frac{1}{\left(1-\sum_{j=0}^n{\bar\zeta}_jz_j\right)^{p}}\right]
=p\left[\frac{z_k}{\left(B(\zeta,z)\right)^{p+1}}\right].
\end{aligned}
\end{equation*}
\end{proof}
\indent
Combining formula \eqref{SmoothnessPartitioned} with the statement of
Lemma~\ref{LocalSmoothness} we obtain the statement of
Proposition~\ref{Smoothness}.
\end{proof}

\indent
Summarizing the results of Lemmas~\ref{BochnerZero} - \ref{BochnerIntegrals}
and of Proposition~\ref{Smoothness} we obtain:

\begin{proposition}\label{Solution} Let $V\subset \C\P^n$ be a reduced
subvariety as in \eqref{Variety}, and let $g_{\alpha}$ be an analytic function on
$U_{\alpha}\subset \C\P^n$ as in Theorem~\ref{HodgeRepresentation}. Let
$\phi=\sum_{\alpha=0}^n\vartheta_{\alpha}\Phi^{(0,q)}$
be a $\bar\partial$-closed residual current of homogeneity zero on $V$.\\
\indent
Then for a fixed $\eta>0$ and an arbitrary $z\in U_{\alpha}$,
such that $|g_{\alpha}(z)|>\eta$, we have the following equality
\begin{multline}\label{SolutionFormula}
\lim_{t\to 0}I_q^{\epsilon(t)}\left[\vartheta_{\beta}\Phi\right](z)
=C(n,q,m)\lim_{t\to 0}\int_{\Gamma^{\epsilon(t)}_J}
\frac{\vartheta_{\beta}(\zeta)\Phi(\zeta)}{\prod_{k=1}^mP_k(\zeta)}\\
\bigwedge\det\left[\frac{\bar z}{B^*(\zeta,z)}\ \frac{\bar\zeta}{B(\zeta,z)}\
\overbrace{Q(\zeta,z)}^{m}\
\overbrace{\frac{d{\bar z}}{B^*(\zeta,z)}}^{q-1}\
\overbrace{\frac{d{\bar\zeta}}{B(\zeta,z)}}^{n-m-q}\right]\wedge\omega(\zeta),
\end{multline}
where $J=(1,\dots,m)$, and $\epsilon(t)$ is an admissible path.\\
\indent
The limit in \eqref{SolutionFormula} is well defined and represents a continuous function
on $\left\{U_{\alpha}:\ |g_{\alpha}(z)|>\eta \right\}$, which is identically zero
if condition \eqref{PhiClosed} is satisfied. If $\phi$ is defined by a $C^{\infty}$ form,
then the limit of the integral in \eqref{SolutionFormula} is a $\C^{\infty}$ form
on $\left\{U_{\alpha}:\ |g_{\alpha}(z)|>\eta \right\}$ for any fixed $\eta$.
\end{proposition}
\begin{proof} We obtain expression \eqref{SolutionFormula} from Lemmas~\ref{BochnerZero}
and \ref{ZeroTerms}, and the rest of the statement from Lemma~\ref{BochnerIntegrals}
and Proposition~\ref{Smoothness}.
\end{proof}

\section{Proof of Theorem~\ref{HodgeRepresentation}.}\label{Proof}

\indent
As the first step in obtaining
formula \eqref{HodgeHomotopy} for $\bar\partial$-closed residual currents
we use Propositions~\ref{HodgeProjector} and \ref{Solution} to obtain the residual limit
of formula \eqref{CurrentFormula}.\\
\indent
Interpreting both sides of \eqref{CurrentFormula} as residual currents we obtain
for a fixed $t$ the equality
\begin{multline*}
\left\langle\phi,\gamma\right\rangle
=\lim_{\tau\to 0}\sum_{\alpha}\int_{T^{\delta(\tau)}_{\alpha}}\vartheta_{\alpha}(z)
\frac{\gamma(z)\wedge\Phi(z)}{\prod_{k=1}^m F^{(\alpha)}_k(z)}\\
=\lim_{\tau\to 0}\sum_{\alpha}\int_{T^{\delta(\tau)}_{\alpha}}
\frac{\vartheta_{\alpha}(z)\bar\partial\gamma(z)}{\prod_{k=1}^m F^{(\alpha)}_k(z)}
\wedge\left(\sum_{\beta}I_q^{\epsilon(t)}\left[\vartheta_{\beta}\Phi\right](z)\right)\\
+\lim_{\tau\to 0}\sum_{\alpha}\int_{T^{\delta(\tau)}_{\alpha}}
\frac{\vartheta_{\alpha}(z)\gamma(z)}{\prod_{k=1}^m F^{(\alpha)}_k(z)}
\wedge\left(\sum_{\beta}I_{q+1}^{\epsilon(t)}
\left[\bar\partial\left(\vartheta_{\beta}\Phi\right)\right](z)\right)\\
+\lim_{\tau\to 0}\sum_{\alpha}\int_{T^{\delta(\tau)}_{\alpha}}
\frac{\vartheta_{\alpha}(z)\gamma(z)}{\prod_{k=1}^m F^{(\alpha)}_k(z)}
\wedge\left(\sum_{\beta}L_q^{\epsilon(t)}\left[\vartheta_{\beta}\Phi\right](z)\right)
\end{multline*}
for an arbitrary $\gamma\in {\cal E}^{(n,n-m-q)}\left(V,{\cal L}\right)$, the differential form
$\Phi(\zeta)=\sum_{\alpha}\vartheta_{\alpha}(\zeta)\Phi_{\alpha}(\zeta)$,
and an admissible path $\left\{\delta_k(\tau)\right\}_1^m$.\\
\indent
Then using for the right-hand side of equality above smoothness of the forms
$I_q^{\epsilon(t)}\left[\vartheta_{\beta}\Phi\right]$,
$I_{q+1}^{\epsilon(t)}\left[\bar\partial\left(\vartheta_{\beta}\Phi\right)\right]$,
and $L_q^{\epsilon(t)}\left[\vartheta_{\beta}\Phi\right]$ with respect to
$z\in U^{\delta(\tau)}$ for fixed $t$ and $\tau\to 0$ we apply 
Theorem 1.8.3 in \cite{CH} (see also \cite{HP4} Prop. 2.2)
and obtain the following equality
\begin{multline}\label{epsilonCurrents}
\left\langle\phi,\gamma\right\rangle
=\lim_{\eta\to 0}
\lim_{\tau\to 0}\int_{\left\{|g_{\alpha}(z)|>\eta\right\}\cap\left\{T^{\delta(\tau)}_{\alpha}\right\}}
\frac{\vartheta_{\alpha}(z)\bar\partial\gamma(z)}{\prod_{k=1}^m F^{(\alpha)}_k(z)}
\wedge\left(\sum_{\beta}I_q^{\epsilon(t)}\left[\vartheta_{\beta}\Phi\right](z)\right)\\
+\lim_{\eta\to 0}\lim_{\tau\to 0}\sum_{\alpha}
\int_{\left\{|g_{\alpha}(z)|>\eta\right\}\cap\left\{T^{\delta(\tau)}_{\alpha}\right\}}
\frac{\vartheta_{\alpha}(z)\gamma(z)}{\prod_{k=1}^m F^{(\alpha)}_k(z)}
\wedge\left(\sum_{\beta}I_{q+1}^{\epsilon(t)}
\left[\bar\partial\left(\vartheta_{\beta}\Phi\right)\right](z)\right)\\
+\lim_{\eta\to 0}\lim_{\tau\to 0}\sum_{\alpha}
\int_{\left\{|g_{\alpha}(z)|>\eta\right\}\cap\left\{T^{\delta(\tau)}_{\alpha}\right\}}
\frac{\vartheta_{\alpha}(z)\gamma(z)}{\prod_{k=1}^m F^{(\alpha)}_k(z)}
\wedge\left(\sum_{\beta}L_q^{\epsilon(t)}\left[\vartheta_{\beta}\Phi\right](z)\right).
\end{multline}

\indent
In the next step we pass to the limit in the right-hand side of \eqref{epsilonCurrents}
as $t\to 0$ and $\eta>0$ is fixed. We use the following lemma to simplify the limit
of the right-hand side of \eqref{epsilonCurrents}.

\begin{lemma}\label{ClosedZero} Let $\phi\in Z_R^{(0,n-m)}\left(V\right)$ be a
$\bar\partial$-closed residual current of homogeneity zero defined by a collection of forms
$\left\{\Phi_{\alpha}\right\}_{\alpha=0}^{n+1}$ satisfying conditions \eqref{ResidualCurrent}
and \eqref{Closed}.\\
\indent
Then for fixed $\eta>0$, $\gamma\in {\cal E}^{(n,0)}\left(V,{\cal L}\right)$, and $\sigma>0$
there exist $\tau, t$, such that
\begin{equation}\label{ZeroTerm}
\left|\int_{\left\{|g_{\alpha}(z)|>\eta\right\}\cap\left\{T^{\delta(\tau)}_{\alpha}\right\}}
\vartheta_{\alpha}(z)\frac{\gamma(z)}{\prod_{k=1}^m F^{(\alpha)}_k(z)}
\wedge\left(\sum_{\beta}I_{q+1}^{\epsilon(t)}
\left[\bar\partial\left(\vartheta_{\beta}\Phi\right)\right](z)\right)\right|<\sigma.
\end{equation}
\end{lemma}
\begin{proof} Because of the choice of $g_{\alpha}$ (see \eqref{gFunctions}) we conclude that
equality \eqref{ZeroTerm} would follow from equality
$$\lim_{t\to 0}\sum_{\beta}I_{q+1}^{\epsilon(t)}
\left[\bar\partial\left(\vartheta_{\beta}\Phi\right)\right](z)=0$$
for $z\in \left\{U_{\alpha}:\ |g_{\alpha}(z)|>\eta\right\}$,
or using formula \eqref{SolutionFormula} from equality
\begin{equation}\label{ZeroPart}
\lim_{t\to 0}\sum_{\beta}\int_{\Gamma^{\epsilon(t)}_J}
\frac{\bar\partial\left(\vartheta_{\beta}(\zeta)\Phi(\zeta)\right)}{\prod_{k=1}^mP_k(\zeta)}
\bigwedge\det\left[\frac{\bar z}{B^*(\zeta,z)}\ \frac{\bar\zeta}{B(\zeta,z)}\
\overbrace{Q(\zeta,z)}^{m}\
\overbrace{\frac{d{\bar z}}{B^*(\zeta,z)}}^{q}\
\overbrace{\frac{d{\bar\zeta}}{B(\zeta,z)}}^{n-m-q-1}\right]\wedge\omega(\zeta)=0,
\end{equation}
where $J=(1,\dots,m)$ and $\left\{\epsilon_k(t)\right\}_{k=1}^m$ is an admissible path.\\
\indent
For the differential form $\Phi$ of homogeneity zero we have
$\Phi_{\alpha}=\Phi\Big|_{U_{\alpha}}=\Phi\Big|_{U_{\beta}}=\Phi_{\beta}$, and therefore
$$\sum_{\beta}\bar\partial\left(\vartheta_{\beta}\Phi_{\beta}\right)
=\sum_{\beta}\vartheta_{\beta}\bar\partial\Phi_{\beta}.$$
But then, using condition \eqref{Closed} and part (ii) of Lemma~\ref{BochnerIntegrals}
we obtain equality \eqref{ZeroPart}.
\end{proof}

\indent
To prove item (i) of Theorem~\ref{HodgeRepresentation} we use
Lemma~\ref{ClosedZero} in equality
\eqref{epsilonCurrents} and obtain for a $\bar\partial$-closed residual current
$\phi\in Z_R^{(0,n-m)}\left(V\right)$ and
an arbitrary $\gamma\in {\cal E}^{(n,0)}\left(V,{\cal L}\right)$ the equality
\begin{equation}\label{HodgeEquality}
\left\langle\phi,\gamma\right\rangle
=\left\langle I_q\left[\phi\right],\bar\partial\gamma\right\rangle
+\left\langle L_q\left[\phi\right],\gamma\right\rangle,
\end{equation}
where
\begin{itemize}
\item
$\left\langle I_q\left[\phi\right],\bar\partial\gamma\right\rangle$
\begin{equation}\label{ICurrent}
=\lim_{\eta\to 0}\lim_{\tau\to 0}\lim_{t\to 0}
\int_{\left\{|g_{\alpha}(z)|>\eta\right\}\cap\left\{T^{\delta(\tau)}_{\alpha}\right\}}
\vartheta_{\alpha}(z)\frac{\bar\partial\gamma(z)}{\prod_{k=1}^m F^{(\alpha)}_k(z)}
\wedge\left(\sum_{\beta}I_q^{\epsilon(t)}\left[\vartheta_{\beta}\Phi\right](z)\right),
\end{equation}
\item
$L_q\left[\phi\right]=0$ for $q=1,\dots, n-m-1$,\vspace{0.15in}
\item
$\left\langle L_{n-m}\left[\phi\right],\gamma\right\rangle$
\begin{equation}\label{LCurrent}
=\lim_{\eta\to 0}\lim_{\tau\to 0}\lim_{t\to 0}
\sum_{\alpha}\int_{\left\{|g_{\alpha}(z)|>\eta\right\}\cap\left\{T^{\delta(\tau)}_{\alpha}\right\}}
\vartheta_{\alpha}(z)\frac{\gamma(z)}{\prod_{k=1}^m F^{(\alpha)}_k(z)}
\wedge\left(\sum_{\beta}L_{n-m}^{\epsilon(t)}\left[\vartheta_{\beta}\Phi\right](z)\right)
\end{equation}
\end{itemize}
with operators $I_q^{\epsilon(t)}$ and $L_{n-m}^{\epsilon(t)}$ defined in
\eqref{SolutionFormula} and \eqref{LResidue} respectively.

\indent
From formula \eqref{LResidue} and
Lemma~\ref{LimitsExistence} it follows that the limit in \eqref{LCurrent} is well defined for
a $\bar\partial$-closed current $\phi$ and an arbitrary
$\gamma\in {\cal E}^{(n,0)}\left(V,{\cal L}\right)$. Then, since the left-hand side is also well
defined for $\gamma\in {\cal E}^{(n,0)}\left(V,{\cal L}\right)$, we obtain that
$\left\langle I_q\left[\phi\right],\bar\partial\gamma\right\rangle$ is also well defined
for a $\bar\partial$-closed residual current $\phi$ and
$\gamma\in {\cal E}^{(n,0)}\left(V,{\cal L}\right)$.\\
\indent
We notice that though the $\bar\partial$-closed current $\phi$ is defined by $C^{\infty}$ forms
satisfying condition \eqref{Closed}, the projection $L_{n-m}[\phi]$ is a residual current defined by the forms analytically depending on $z,{\bar z}$.\\
\indent
To prove item (ii) we use formula \eqref{LOperator} to obtain that operator $L_q$ is not zero only for $q=n-m$,
and therefore formula \eqref{ICurrent} gives a solution $I_q\left[\phi\right]$
of the $\bar\partial$-equation
$$\bar\partial\psi=\phi$$
for a $\bar\partial$-closed residual current $\phi^{(0,q)}$ of homogeneity zero for $q<n-m$.
Smoothness of $I_q\left[\phi\right](z)$ on $U_{\alpha}\setminus V^{\prime}_{\alpha}$
for smooth $\left\{\Phi_{\alpha}\right\}_{\alpha=0}^n$ follows from Proposition~\ref{Solution}.\\
\indent
For $q=n-m$ we have a nontrivial cohomology group $H_R^{n-m}\left(V,\ {\cal O}_V\right)$.
In the Proposition below we prove the necessary and sufficient condition from item (iii) in
Theorem~\ref{HodgeRepresentation} for a $\bar\partial$-closed residual current to be exact.

\begin{proposition}\label{Decomposition} Let $V\subset \C\P^n$ be a reduced
complete intersection subvariety as in \eqref{Variety} satisfying conditions of
Theorem~\ref{HodgeRepresentation}.
Then a $\bar\partial$-closed residual current $\phi\in Z_R^{(0,n-m)}\left(V\right)$
of homogeneity zero is $\bar\partial$-exact, i.e. there exists a current
$\psi\in C^{(0,n-m-1)}\left(V\right)$
such that $\bar\partial\psi=\phi$, iff condition \eqref{HodgeCondition} is satisfied.
\end{proposition}
\begin{proof} Sufficiency of condition \eqref{HodgeCondition} immediately follows from
equality \eqref{HodgeEquality}. On the other hand, if $\phi=\bar\partial\psi$ for a current
$\psi\in  C^{(0,n-m-1)}(V)$ of homogeneity zero, then we have
equality
\begin{equation*}\langle\phi,\gamma\rangle
=\langle\psi,\bar\partial\gamma\rangle
\end{equation*}
satisfied for an arbitrary $\gamma\in {\cal E}^{(n,0)}\left(V,{\cal L}\right)$.\\
\indent
Applying the last equality to differential forms
$$\gamma^r_z(\zeta)=\langle{\bar z}\cdot\zeta\rangle^r
\wedge\det\left[{\bar z}\ \overbrace{Q(\zeta,z)}^{m}\
\overbrace{d{\bar z}}^{n-m}\right]\wedge\omega(\zeta),$$
and using Lemma~\ref{LimitsExistence} and holomorphic dependence
of the forms $\gamma^r_z$ on $\zeta$ we obtain equality
\begin{multline*}
\lim_{t\to 0}L_{n-m}^{\epsilon(t)}\left[\Phi\right](z)
=\sum_{0\leq r\leq d-n-1}C(n,m,d,r)\lim_{t\to 0}\int_{\left\{|\zeta|=1,\
\left\{|P_k(\zeta)|=\epsilon_k(t)\right\}_{k=1}^m\right\}}
\frac{\gamma^r_z(\zeta)\wedge\Phi(\zeta)}{\prod_{k=1}^mP_k(\zeta)}\\
=\sum_{0\leq r\leq d-n-1}C(n,m,d,r)
\cdot\langle\psi,\bar\partial_{\zeta}\gamma^r_z\rangle=0
\end{multline*}
for an arbitrary $z$ such that $|g_{\beta}(z)|>\eta$.\\
\indent
Using this equality in \eqref{LCurrent} we obtain the necessity of condition
\eqref{HodgeCondition}.
\end{proof}

\indent
This concludes the proof of Theorem~\ref{HodgeRepresentation}.


\end{document}